\documentclass[12pt]{amsproc}
\usepackage[left=1in, right=1in, top=1in, bottom=1in]{geometry}
\usepackage[authoryear]{natbib} 
\usepackage{amsmath,amsthm,amsfonts,amssymb}
\usepackage[T1]{fontenc}
\usepackage{mathtools}
\usepackage[dvipsnames]{xcolor}
\usepackage[colorlinks=true,citecolor=blue,urlcolor=blue, linkcolor=blue]{hyperref}
\usepackage{blkarray}
\usepackage{tikz,tikz-cd}
\usepackage[noabbrev]{cleveref}
\usepackage{bbm}
\usepackage{algorithm}
\usepackage{algpseudocode}
\usetikzlibrary{chains,matrix,scopes,fit,decorations,positioning,shapes,arrows,decorations.markings}
\usepackage{subcaption}
\usepackage{graphicx}

\newcommand\C{\mathbb{C}}

\newcommand\D{\mathrm{D}}

\newcommand\cI{\mathcal{I}}
\newcommand\cJ{\mathcal{J}}
\newcommand\cS{\mathcal{S}}
\newcommand\cT{\mathcal{T}}

\newcommand\codim{\operatorname{codim}}
\newcommand\rO{\mathrm{O}}
\newcommand\rSO{\mathrm{SO}}
\newcommand\rSP{\mathrm{SP}}
\newcommand\T{\mathcal{T}}

\newcommand\ba{\mathbf{a}}
\newcommand\bi{\mathbf{i}}

\newcommand\s{\mathbf{s}}
\newcommand\bt{\mathbf{t}}
\newcommand\x{\mathbf{x}}
\newcommand\y{\mathbf{y}}
\newcommand\z{\mathbf{z}}

\newcommand\rank{\mathrm{rank}}

\newcommand\pmi{\mathrm{pmi}}
\newcommand\mi{\mathrm{mi}}

\definecolor{DarkGreen}{rgb}{0.0, 0.5, 0.0}

\newcommand\E{\mathbb{E}}

\newcommand\R{\mathbb{R}}

\newcommand\<{\langle}
\renewcommand\>{\rangle}
\newcommand\off{\rm off}

\newcommand\independent{\protect\mathpalette{\protect\independenT}{\perp}}
\def\independenT#1#2{\mathrel{\rlap{$#1#2$}\mkern2mu{#1#2}}}

\DeclareMathOperator*{\argmin}{arg\,min}
\DeclareMathOperator{\diag}{diag}

\DeclareMathOperator{\Cov}{Cov}
\DeclareMathOperator{\Var}{Var}

\theoremstyle{plain}
\newtheorem{thm}{thm}[section]
\newtheorem{theorem}[thm]{Theorem}
\newtheorem{proposition}[thm]{Proposition}
\newtheorem{lemma}[thm]{Lemma}
\newtheorem{corollary}[thm]{Corollary}

\theoremstyle{definition}
\newtheorem{definition}[thm]{Definition}
\newtheorem{example}[thm]{Example}
\newtheorem{assumption}[thm]{Assumption}

\theoremstyle{remark}
\newtheorem{remark}[thm]{Remark}

\AddToHook{env/theorem/begin}{\crefalias{thm}{theorem}}
\AddToHook{env/proposition/begin}{\crefalias{thm}{proposition}}
\AddToHook{env/lemma/begin}{\crefalias{thm}{lemma}}
\AddToHook{env/corollary/begin}{\crefalias{thm}{corollary}}
\AddToHook{env/conjecture/begin}{\crefalias{thm}{conjecture}}
\AddToHook{env/definition/begin}{\crefalias{thm}{definition}}
\AddToHook{env/example/begin}{\crefalias{thm}{example}}
\AddToHook{env/assumption/begin}{\crefalias{thm}{assumption}}
\AddToHook{env/remark/begin}{\crefalias{thm}{remark}}

\crefname{thm}{theorem}{theorems}
\Crefname{thm}{Theorem}{Theorems}
\crefname{theorem}{theorem}{theorems}
\Crefname{theorem}{Theorem}{Theorems}
\crefname{proposition}{proposition}{propositions}
\Crefname{proposition}{Proposition}{Propositions}
\crefname{lemma}{lemma}{lemmas}
\Crefname{lemma}{Lemma}{Lemmas}
\crefname{corollary}{corollary}{corollaries}
\Crefname{corollary}{Corollary}{Corollaries}
\crefname{conjecture}{conjecture}{conjectures}
\Crefname{conjecture}{Conjecture}{Conjectures}
\crefname{definition}{definition}{definitions}
\Crefname{definition}{Definition}{Definitions}
\crefname{example}{example}{examples}
\Crefname{example}{Example}{Examples}
\crefname{assumption}{assumption}{assumptions}
\Crefname{assumption}{Assumption}{Assumptions}
\crefname{remark}{remark}{remarks}
\Crefname{remark}{Remark}{Remarks}

\title[Beyond independent component analysis]{Beyond independent component analysis:\\identifiability and algorithms}
\author{\'Alvaro Ribot}
\address{\'{A}lvaro Ribot, Harvard University}
\email{aribotbarrado@g.harvard.edu}
\author{Anna Seigal}
\address{Anna Seigal, Harvard University}
\email{aseigal@seas.harvard.edu}
\author{Piotr Zwiernik}
\address{
Piotr Zwiernik, 
Universitat Pompeu Fabra and Barcelona School of Economics}\email{piotr.zwiernik@upf.edu}

\begin{document}

\begin{abstract}
Independent Component Analysis (ICA) is a classical method for recovering latent variables with useful identifiability properties. 
For independent variables, cumulant tensors are diagonal; relaxing independence yields tensors whose zero structure generalizes diagonality.
These models have been the subject of recent work in non-independent component analysis. We show that pairwise mean independence answers the question of how much one can relax independence: it is identifiable, any weaker notion is non-identifiable, and it contains the models previously studied as special cases. Our results apply to distributions with the required zero pattern at any cumulant tensor.
We propose an algebraic recovery algorithm based on least-squares optimization over the orthogonal group. Simulations highlight robustness: enforcing full independence can harm estimation, while pairwise mean independence enables more stable recovery. These findings extend the classical ICA framework and provide a rigorous basis for blind source separation beyond independence.
\end{abstract}

\maketitle

\section{Introduction}

Independent component analysis (ICA) is a tool for blind source separation. It turns linear mixtures into interpretable sources and underpins methods in signal processing, neuroscience, and econometrics \cite{back1997first, cardoso1998blind, makeig1995independent}.
Blind source separation 
seeks to estimate $A$ and $\s$ from observations of $\x =A\s$, where $\x$ and $\s$ are $n$-dimensional random vectors, the mixing matrix~$A\in \R^{n\times n}$ is fixed and invertible but unknown, and $\s$ is assumed to have mean zero 
and uncorrelated entries but is otherwise unknown.
With no further restrictions on the distribution of $\s$, the solution is not unique. ICA assumes that the entries of $\s$ are independent.
Then, the matrix $A$ is unique up to scaling and permutation of columns, provided at most one source is Gaussian \cite{comon1994independent}. 

In econometrics and other applications, there is growing interest in relaxing the assumption of independent sources \cite{garrote2024cumulant,hyvarinen2001topographic,jiang2025identification,lee2020testing,wooldridge1995selection}; see \cite{mesters2022non} for a detailed discussion. The idea is that full independence of the source variables is too strong. Relaxations of independence offer a more realistic description of real-world systems. 
However, relax the assumptions too much and $\x = A\s$ is no longer identifiable. The goal is to relax independence while preserving identifiability of $A$ and $\s$ \cite{mesters2022non,garrote2024cumulant}. 

\begin{remark}[Identifiability]
In the component analysis models $\x = A \s$ we consider,
rescaling and relabeling the source variables does not affect membership in the model. Hence, at best, we can recover $\s$ 
up to rescaling and reordering its coordinates,
with a corresponding scaling and permutation of the columns of $A$. We call a model identifiable if 
we can recover any  sufficiently general sources and mixing matrix, up to this relabeling and rescaling.
Sufficiently general, here, means that a certain polynomial in the entries of a higher-order cumulant of $\s$ does not vanish and that the matrix $A$ is invertible.
\end{remark}

One popular relaxation replaces independence—the assumption that conditioning gives no information on the distribution—by mean independence, the weaker assumption that conditioning gives no information on the expectation. Concretely, two random variables are \emph{mean independent} when knowing one of them does not affect the expected value of the other, i.e., $\E(x\mid y)=\E(x)$.  An illustration comes from the weak form of the Efficient-Market Hypothesis (EMH) \cite{fama1970efficient}: for excess returns $r_t$ and the information set $\mathcal F_{t-1}$ known at $t-1$, $\E(r_t\mid\mathcal F_{t-1})=0$. If a lagged sentiment index $s_{t-1}$ is part of $\mathcal F_{t-1}$, then~$\E(r_t\mid s_{t-1})=\E(r_t)=0$. Related conditions appear in classical measurement error, where the error term $u$ is assumed mean independent of the regressor $x$ ($\E(u\mid x)=0$) \cite{wooldridge2010econometric}, and in randomized controlled trials, where random assignment implies $\E[Y(0)\mid D]=\E[Y(0)]$ \cite[Chapter~1]{imbens2015causal}.
    
\begin{definition}
    For random variables $x$ and $y$, 
    $x$ is \emph{mean independent} of $y$ if $\E(x \mid y) = \E(x)$.
    A random vector $\x$ is \emph{pairwise mean independent} if $\E(x_i\mid x_j)=\E(x_i)$
    for all~$i\neq j$.
    A random vector $\x$ is \emph{mean independent} if $\E(x_i\mid\x_{\setminus i})=\E(x_i)$ for all $i$, where $\x_{\setminus i}$ is the vector~$\x$ with coordinate $x_i$ removed.
\end{definition}

Independence implies mean independence, mean independence implies pairwise mean independence, and all three imply uncorrelatedness. All notions are distinct. 
Mean independence and pairwise mean independence appear in~\cite{garrote2024cumulant,jiang2025identification,lee2020testing,mesters2022non,wooldridge1995selection,wooldridge2010econometric}, 
with varying terminology. In \cite{jiang2025identification,garrote2024cumulant,mesters2022non} they study mean independence of a random vector and the structure of its moments/cumulants. 
The papers \cite{wooldridge1995selection,wooldridge2010econometric} study conditional mean independence (e.g., $\E(x \mid y, z) = \E(x \mid y)$). In \cite{lee2020testing}, mean independence of two random variables is called conditional mean independence.

Our first main result generalizes the main result of \cite{comon1994independent} from independence to pairwise mean independence. It shows the identifiability of component analysis with pairwise mean independent source distributions.

\begin{theorem}[Identifiability of PMICA]\label{thm:PMICA}
    Consider the model $\x = A\s$ where $A \in \R^{n \times n}$ is invertible and $\s$ is a sufficiently general pairwise mean independent random vector. Then $A$ is identifiable from~$\x$ (up to permutation and scaling of columns).
\end{theorem}

Two works lead to \Cref{thm:PMICA}. 
First, \cite{mesters2022non} introduced non-independent component analysis and showed how identifiability can be studied via 
moment or cumulant tensors. They studied independence via diagonal tensors and obtained first identifiability results for other zero patterns.
Second, \cite{ribot2025orthogonal} showed that a generic symmetric tensor with an orthogonal basis of eigenvectors has a unique such basis (up to signs). The existence of an orthogonal basis of eigenvectors corresponds to a prescribed zero pattern. 
We show that this zero pattern characterizes pairwise mean independent (PMI) distributions  and that sufficiently general PMI distributions have cumulants that are sufficiently general as tensors.

Sufficiently general in \Cref{thm:PMICA} means the non-vanishing of a certain polynomial in the entries of a higher-order cumulant of $\s$. A study of these genericity conditions is another main contribution of this paper.
In classical ICA, a simple genericity rule guarantees identifiability: at most one Gaussian source. 
Under PMI the genericity is more subtle: non-Gaussianity does not suffice. We provide low-order, checkable criteria: 
for $d=3$, uniqueness holds if and only if at most one third cumulant $\kappa_3(s_i)$ is zero; for $d=4$, uniqueness holds if and only if the fourth-order cumulants $\kappa_4(s_i)$ are all distinct (Theorem~\ref{thm:genericity-Vpmi}). For $d\ge5$ we obtain polynomial non-vanishing conditions. The conditions are illustrated in Figure~\ref{fig:genericity_binary}. 

\begin{figure}[ht]
  \centering
  \captionsetup[subfigure]{labelformat=empty}

  \begin{subfigure}[b]{0.115\textwidth}
    \includegraphics[width=\linewidth]{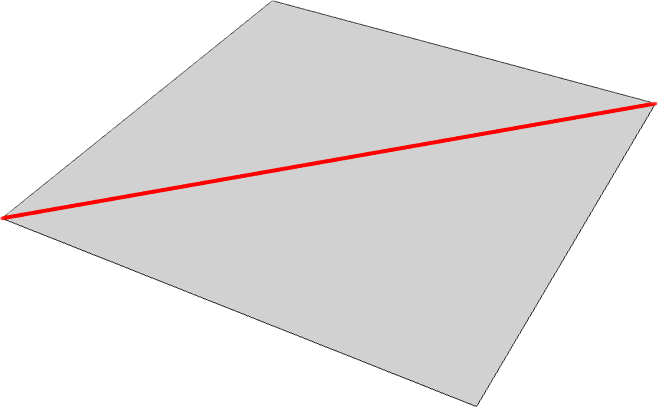}
    \caption{$d=2$}
  \end{subfigure}
  \hfill
  \begin{subfigure}[b]{0.115\textwidth}
    \includegraphics[width=\linewidth]{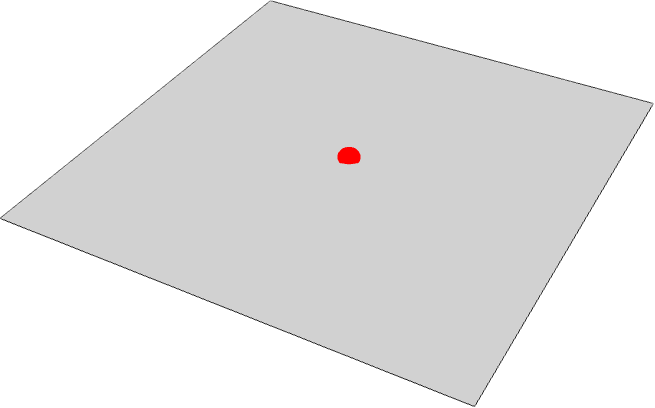}
    \caption{$d=3$}
  \end{subfigure}
  \hfill
  \begin{subfigure}[b]{0.115\textwidth}
    \includegraphics[width=\linewidth]{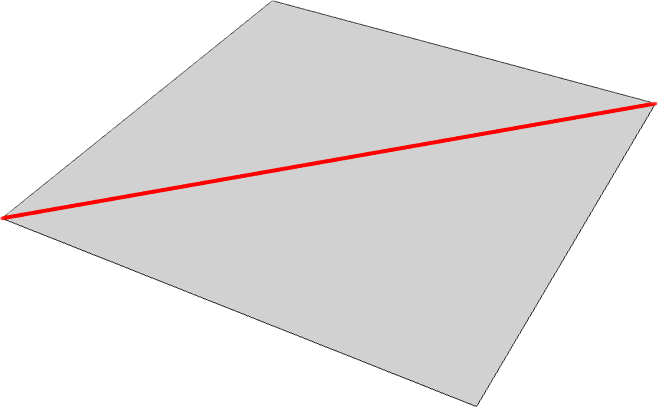}
    \caption{$d=4$}
  \end{subfigure}
  \hfill
  \begin{subfigure}[b]{0.115\textwidth}
    \includegraphics[width=\linewidth]{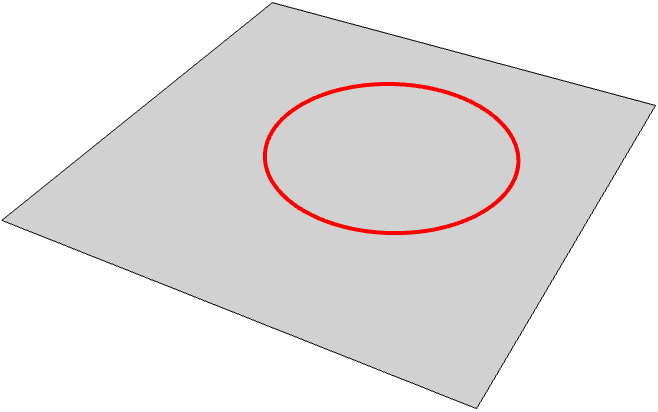}
    \caption{$d=5$}
  \end{subfigure}
  \hfill
  \begin{subfigure}[b]{0.115\textwidth}
    \includegraphics[width=\linewidth]{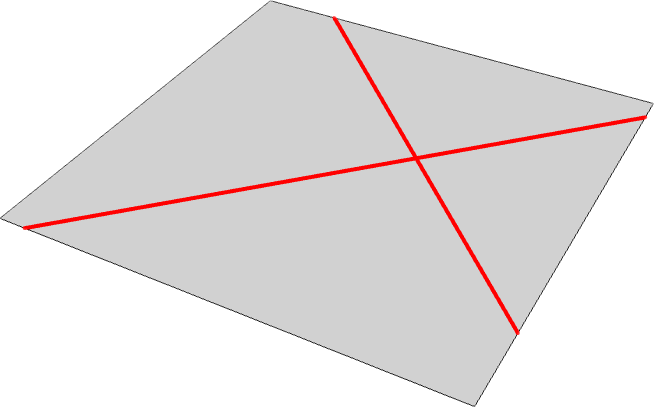}
    \caption{$d=6$}
  \end{subfigure}
  \hfill
  \begin{subfigure}[b]{0.115\textwidth}
    \includegraphics[width=\linewidth]{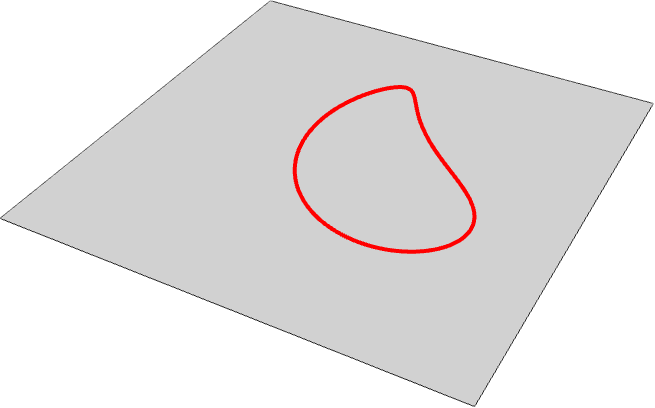}
    \caption{$d=7$}
  \end{subfigure}
  \hfill
  \begin{subfigure}[b]{0.115\textwidth}
    \includegraphics[width=\linewidth]{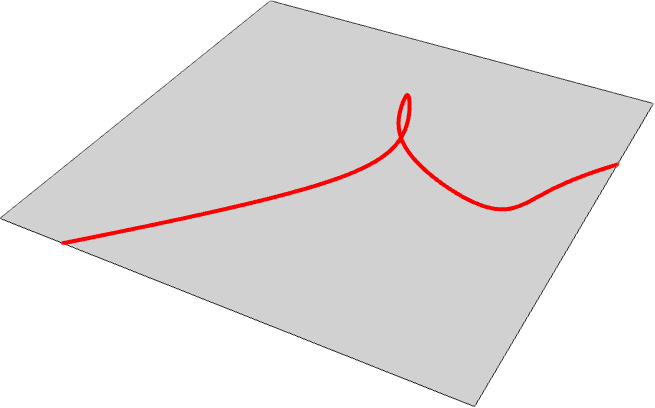}
    \caption{$d=8$}
  \end{subfigure}
  \hfill
  \begin{subfigure}[b]{0.115\textwidth}
    \includegraphics[width=\linewidth]{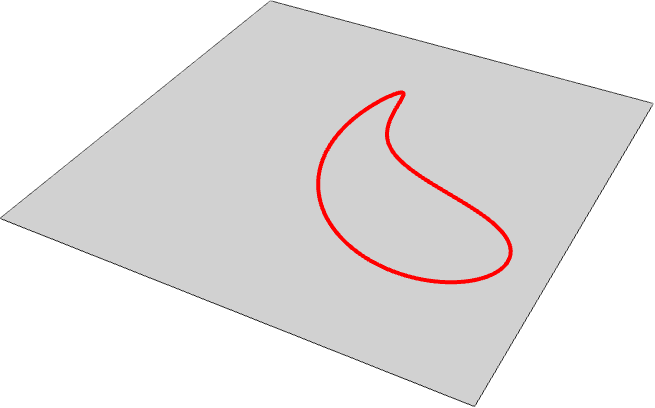}
    \caption{$d=9$}
  \end{subfigure}

  \caption{
  Genericity conditions in the space of $d$-th order cumulant tensors of pairwise mean independent distributions. The gray plane is the this space, and the red loci are tensors with a non-unique orthogonal basis of eigenvectors. The pictures are three-dimensional slices explained in \Cref{remark:genericity-large-d}.  }
  \label{fig:genericity_binary}
\end{figure}

For an $n$-dimensional random vector $\s$, the $d$-th cumulant is a symmetric tensor, an array of format $n \times \cdots \times n$ ($d$ times) whose entries are unchanged under permuting indices. 
We denote the space of such symmetric tensors by $S^d(\mathbb{R}^n)$ and the $d$th cumulant of $\s$ by $\kappa_d(\s)$.
Independence  implies diagonal cumulant tensors; see, e.g., ~\cite{zwiernik2015semialgebraic}.  
That is, the cumulants of independent variables $\s$ are a linear subspace of $S^d(\R^n)$
consisting of diagonal tensors 
\begin{equation}\label{eq:Vdiag}
    V_{\diag} \coloneqq \{ \T \in S^d(\R^n) \mid \T_{i_1 \dots i_d} \neq 0 \text{ only if } i_1 = \dots = i_d\}.
\end{equation}
There is a corresponding linear subspace for pairwise mean independence: 
\begin{equation}\label{eq:Vpmi}
    V_{\pmi} \coloneqq \{\T \in S^d(\R^n) \mid \T_{ij\dots j} = 0 \text{ for all } i \neq j\}.
\end{equation}
For a linear space $V\subset S^d(\R^n)$, we will write $V^{d,n}$ when the ambient space is not clear from the context.
We show in Theorem~\ref{thm:V_pmi} that the distributions whose cumulants lie in $V_{\pmi}$ for every $d$ are exactly the PMI distributions. 
When $d=2$ the two spaces coincide for all $n$, but they differ for higher~$d$: $\dim V_{\diag}=n$ for all $d$,
whereas $\dim V_{\pmi}=\binom{n+d-1}{d}-2 \binom{n}{2}$ for $d\ge3$. 
The set $V_{\pmi}$ was introduced in \cite[Conjecture 5.17]{mesters2022non} and used in \cite{ribot2025orthogonal} to prove generic uniqueness of an orthogonal basis of eigenvectors for tensors. See \Cref{tab:results-glance} for a summary of our contributions.

\begin{table}[htbp]
\centering
\renewcommand{\arraystretch}{1.2}
\small
\begin{tabular}{l c c}
\hline
\textbf{Assumption on sources} & \textbf{Cumulants of sources} & \textbf{Generic identifiability} \\
\hline
Independence & $V_{\diag}$ & yes \\
Between independence and PMI& $V$ with $V_{\diag}\subseteq V\subseteq V_{\pmi}$ & yes \\
Pairwise mean independence (PMI) & $V_{\pmi}$ & yes \\

Beyond PMI & $V$ with $V_{\pmi} \subsetneq V$ & no \\
\hline
\end{tabular}
\caption{Results at a glance. For $\x = A \s$, $A$ is generically identifiable if $\s$ is PMI (\Cref{thm:PMICA}) or if its $d$-th order cumulant tensor lies in~$V$
with $V_{\diag} \subseteq V \subseteq V_{\pmi}$ (\Cref{thm:general-ca}). 
Identifiability is lost if a mean independence condition is dropped (\Cref{thm:pmi-maximal}).
Linear spaces $V_{\diag}$ and $V_\pmi$ are the cumulant tensors of independent and pairwise mean independent distributions, respectively; see~\eqref{eq:Vdiag} and~\eqref{eq:Vpmi}.
}
\label{tab:results-glance}
\end{table}

Motivated by our identifiability results, we propose an algebraic algorithm to recover pairwise mean independent sources from cumulant tensors. The problem is formulated as a least-squares optimization on the orthogonal group, which we attempt to solve via Riemannian gradient descent (RGD). See \Cref{alg:cap}.

\begin{algorithm}[htbp]
\caption{\texttt{RGD-PMICA}: Finding PMI sources}\label{alg:cap}
\begin{algorithmic}
\Require Data matrix $X \quad \textcolor{gray}{\in \R^{N \times n}}$ \Comment{$N = $ \# samples, $n = $ \# features}
\State $X_w \gets \texttt{Whitening}(X) \quad \textcolor{gray}{ \in \R^{N \times n}}$ \Comment{Center the mean and decorrelate}
\State $\widehat{\kappa} \gets \texttt{CumulantTensor4}(X_w)$ \quad \textcolor{gray}{$\in S^4(\R^n)$} \Comment{Estimate fourth-order cumulant tensor}
\State $\widehat{Q} \gets \argmin_{Q \in \rO(n)} \sum_{i \neq j} [Q^\top \bullet \widehat{\kappa}]^2_{ijjj}$ \quad \textcolor{gray}{$\in \R^{n \times n}$} \Comment{Minimize distance to $V_{\pmi}$ with RGD}  
\Ensure $S \gets X_w \widehat{Q} \quad \textcolor{gray}{\in \R^{N \times n}}$ \Comment{PMI sources}
\end{algorithmic}
\end{algorithm}

A practical route to identifiability is to impose a \emph{generative} structure that forces certain entries of $\kappa_d(\s)$ to vanish. Popular examples are \emph{topographic ICA} \cite{hyvarinen2001topographic} and, more generally, \emph{correlated-energy} models $s_i=\sigma_i\varepsilon_i$ with independent $\varepsilon_i$ and exogenous scales $\sigma_i$, which imply PMI and hence $\kappa_d(\s)\in V_{\pmi}$ (see \Cref{sec:correlation-energies}).

The next result shows that any such family inherits generic identifiability from a single $d$-th order cumulant once $\kappa_d(\s)$ is sufficiently general in the corresponding linear space. As in Theorem~\ref{thm:PMICA}, a source distribution $\s$ is sufficiently general if the sources $\s = (s_1, \ldots, s_n)$ have at most one third cumulant $\kappa_3(s_i)$ equal to zero, or if all fourth order cumulants $\kappa_4(s_i)$ are distinct.

\begin{theorem}\label{thm:general-ca}
Consider the model $\x = A\s$ where $A \in \R^{n \times n}$ is invertible and $\s$ is a random vector with identity covariance matrix such that, for some $d \geq 3$, $\kappa_d(\s)$ is sufficiently general in $V$, where $V_{\diag} \subseteq V \subseteq V_{\pmi}$. Then, $A$ is identifiable from~$\kappa_d(\x)$ (up to permutation and sign-flip of columns).
\end{theorem}

We also show that PMI is maximal for identifiability: models are generically identifiable under pairwise mean independence, but generically unidentifiable once any such assumption is dropped. Each mean independence assumption imposes one zero restriction on each cumulant tensor.

\begin{theorem}\label{thm:pmi-maximal}
    Consider the model $\x = A\s$ where $A \in \R^{n \times n}$ is invertible and $s_i$ is mean independent of $s_j$ for all pairs $i \neq j$ except one. Then $A$ is generically unidentifiable from $\kappa_d(\x)$ for any $d \geq 3$ (up to permutation and scaling of columns).
\end{theorem}

The rest of this paper is organized as follows. We characterize the cumulants of pairwise mean independent distributions in Section~\ref{sec:kappa_pmi}. We prove \Cref{thm:PMICA} in \Cref{sec:pmica}. 
We prove \Cref{thm:general-ca} and describe what it means for a distribution or cumulant tensor to be sufficiently general in Section~\ref{sec:proof-uniqueness-symmetric}. We prove \Cref{thm:pmi-maximal} and discuss other relaxations of independence under which identifiability holds in Section~\ref{sec:other_relaxations}.
We investigate the consistency and sample complexity of our minimum-distance estimator in \Cref{sec:estimator}. We analyze when ICA methods recover PMI distributions in \Cref{sec:local_optima}.
Our numerical experiments are in \Cref{sec:numerical}.

\section{Cumulants of pairwise mean independent distributions}
\label{sec:kappa_pmi}

This section links pairwise mean independence to the linear space $V_\pmi$ in \eqref{eq:Vpmi}.
Given an $n$-dimensional random vector $\z$, let $M_\z(\bt) = \E(e^{\bt^\top \z})$ and $K_\z(\bt) = \log M_\z(\bt)$ be the moment-generating function and cumulant-generating function of $\z$, respectively. The $d$-th moment tensor is $\mu_d(\z)=\E(\z^{\otimes d})$. Its entries are the $d$-th order partial derivatives of $M_\z(\bt)$ evaluated at $\bt=0$. The $d$-th cumulant tensor
$\kappa_d(\z)$ has entries $\kappa_d(\z)_{i_1, \dots, i_d} = \frac{\partial^d}{\partial t_{i_1} \cdots \partial t_{i_d}} K_\z(\bt) \big\rvert_{\bt = 0}$, whenever the corresponding partial derivative exists. 

The second moment and cumulant are symmetric matrices. The higher-order moments and cumulants are symmetric tensors:
a tensor $\T$ is symmetric if $\T_{i_1, \dots, i_d} = \T_{i_{\sigma(1)}, \dots, i_{\sigma(d)}}$ for any $i_1, \dots, i_d \in [n]$ and any permutation $\sigma \in S_d$. If $d=2$, this recovers the definition of symmetric matrices. Let $S^d(\R^n)$ be the set of real symmetric $n \times \cdots \times n$ tensors of order $d$.

We write $V_\pmi^{d,n}$ instead of $V_\pmi$ to make the tensor format explicit: $V_\pmi^{d,n} \subset S^d(\R^n)$.
For $d=2$, $V_\pmi^{2,n}=V_{\diag}^{2,n}$ is the space of diagonal $n\times n$ matrices.

\begin{theorem}
If $\z = (z_1, \dots, z_n)$ has independent coordinates, then $\kappa_d(\z) \in V_{\diag}^{d,n}$ for all $d\geq 2$ for which the corresponding moments of order $d$ exist. Conversely, if all off-diagonal entries of $\kappa_d(\z)$ vanish for every $d\ge2$ and the cumulant-generating function $K_\z(\bt)$ is finite in a neighborhood of $0$, then the coordinates of $\z$ are independent.
\end{theorem}

\begin{remark}
The forward implication is classical (vanishing mixed cumulants under independence). For the converse, finiteness of $K_\z$ near $0$ implies that $K_\z(\bt)=\sum_{i=1}^n K_{z_i}(t_i)$, hence $M_\z(\bt)=\prod_{i=1}^n M_{z_i}(t_i)$ in a neighborhood of $0$, which characterizes independence.
\end{remark}

The next result shows that $V_\pmi$, defined in \eqref{eq:Vpmi}, plays for pairwise mean independence the role that $V_{\diag}$ plays for independence.
One direction was given in \cite[Proposition 3.3]{garrote2024cumulant}. We give a detailed explanation and also show the converse direction.

\begin{theorem}\label{thm:V_pmi}
If $\z = (z_1, \dots, z_n)$ is pairwise mean independent then $\kappa_d(\z) \in V_\pmi^{d,n}$ for all $d\geq 2$ for which the corresponding moments of order $d$ exist. Conversely, if $\kappa_d(\z)\in V_\pmi^{d,n}$ for all $d\ge2$ and $K_\z(\bt)$ is finite in a neighborhood of $0$, then $\z$ is pairwise mean independent.
\end{theorem}

\begin{proof}
Fix $d\ge2$ such that all moments of order $d$ exist. Then all moments and cumulants up to order $d$ exist. The $(i,j,\ldots,j)$ entry of the $d$-th order moments tensor $\mu_d(\z)$ equals $\E(z_i z_j^{\,d-1})$. By the tower property,
\begin{equation}\label{eq:wedgepi0}
    \E(z_i z_j^{\,d-1})\;=\;\E \big(z_j^{\,d-1}\,\E(z_i \mid z_j)\big)\;=\;\E(z_i)\E(z_j^{\,d-1})\qquad(i\ne j).
\end{equation}
Let $\Pi_d$ be the partition lattice of $\{1,\dots,d\}$ and, for $\pi\in \Pi_d$, denote by $|\pi|$ the number of blocks of $\pi$. The cumulant-moment relation is
\begin{equation}\label{eq:KinM}
    \kappa_d(\z)_{i_1,\ldots, i_d}\;=\;\sum_{\pi\in \Pi_d}(-1)^{|\pi|-1}(|\pi|-1)!\ \prod_{B\in \pi} \E  \Big(\prod_{k\in B}z_{i_k}\Big),
\end{equation}
see, for example, \cite[§4.2.1]{zwiernik2015semialgebraic}.
Specialize to $(i_1,\dots,i_d)=(i,j,\dots,j)$ with $i\ne j$ and let $\pi_0:=1\,|\,2\cdots d$ be a partition with two blocks $\{1\}$ and $\{2,\ldots,d\}$. For any $\pi\in\Pi_d$, write $\pi\wedge\pi_0$ for the partition obtained by splitting the block of $\pi$ that contains $1$ into $\{1\}$ and the rest. By \eqref{eq:wedgepi0},
\[
\prod_{B\in \pi}\E \Big(\prod_{k\in B} z_{i_k}\Big)\;=\;\prod_{B\in \pi\wedge \pi_0}\E \Big(\prod_{k\in B} z_{i_k}\Big).
\]
Hence
\begin{align*}
\kappa_d(\z)_{i,j,\ldots,j}
&=\sum_{\pi\in \Pi_d}(-1)^{|\pi|-1}(|\pi|-1)!\prod_{B\in \pi\wedge \pi_0}\E \Big(\prod_{k\in B} z_{i_k}\Big)\\
&=\sum_{\delta\le \pi_0}\Bigg(\sum_{\pi:\ \pi\wedge\pi_0=\delta}(-1)^{|\pi|-1}(|\pi|-1)!\Bigg)\prod_{B\in \delta}\E \Big(\prod_{k\in B} z_{i_k}\Big),
\end{align*}
and the inner sum vanishes by M\"obius inversion on $\Pi_d$ \cite[Lemma 4.19]{zwiernik2015semialgebraic}, so~$\kappa_d(\z)\in V_\pmi^{d,n}$.

For the converse, we use a standard $L^2$ projection argument. Here $L^2$ denotes the set of all square integrable functions with respect to the fixed probability space and with the standard inner product. Cumulants are invariant under mean shifts, so assume $\E z_i=0$. Define the Hilbert space of square-integrable functions of $W$ by
\[
\mathcal H(W)\ :=\ \{\,h(W):\ h \text{ measurable and } \E[h(W)^2]<\infty\,\}.
\]
Here, $\mathcal H(W)\subseteq L^2$ and the conditional expectation $\E(\cdot|W)$ gives the orthogonal projection from $L^2$ to $\mathcal H(W)$; see, e.g., Exercise~34.12 in \cite{billingsley1995probability}. Let $g(z_j):=\E(z_i\mid z_j)\in\mathcal H(z_j)$. The vanishing of all mixed entries $\kappa_{k+1}(\z)_{i,j,\ldots,j}=0$ for $k\ge1$ implies (by \eqref{eq:KinM} as above) that
\[
\E\big(g(z_j)\,z_j^{\,k}\big)=0\qquad\text{for all }k\ge0.
\]
Since $K_\z$ is finite near $0$, $\{z_j^k:k\ge0\}$ is dense in $\mathcal H(z_j)$ (polynomial density; see \cite[Cor.~2.3.3]{akhiezer2020classical}). Hence $g(z_j)\in \mathcal H(z_j)$ is orthogonal to a dense subset of $\mathcal H(z_j)$, so $g(z_j)=0$ in $L^2$, i.e., $\E(z_i\mid z_j)=0$ almost surely. This holds for all $i\ne j$, proving pairwise mean independence.
\end{proof}

\begin{corollary}\label{cor:momentssuff}
    Theorem~\ref{thm:V_pmi} remains true with $\mu_d$ in place of $\kappa_d$ provided $\E(\z)=0$.
\end{corollary}

\section{Pairwise Mean Independent Component Analysis}
\label{sec:pmica}

We prove Theorem~\ref{thm:PMICA}; i.e., we show that we can identify $A$ in the model $\x=A\s$ if~$\s$ is pairwise mean independent and sufficiently general. 
This gives the identifiability of Pairwise Mean Independent Component Analysis (PMICA). Our approach is to relate the identifiability of PMICA to the study of orthogonal eigenvectors of tensors~\cite{ribot2025orthogonal}.

We first explain how to reduce from an invertible matrix to an orthogonal matrix in our component analysis model, a standard procedure that sometimes goes by the name of whitening. 
Given a random vector $\y$, we define its transformation $\y_w$ to be $Q \Lambda^{-1/2} Q^\top (\y - \E (\y))$ where $Q \Lambda Q^\top$ is the eigendecomposition of the covariance matrix $\Cov(\y) = \E ((\y - \E (\y)) (\y - \E (\y))^\top)$. Then $\E (\y_w) = 0$ and $\E (\y_w \y_w^\top) = I$. This shifts and rescales the variables so that they have mean zero and variance one. We can assume, without loss of generality, that $\Cov(\s) = I$, as follows. If $\s$ is PMI, then it has uncorrelated entries, and scaling $\s$ does not change membership in the PMICA model because $V_{\pmi}$ is given by zero restrictions and the scalars can be absorbed by the mixing matrix $A$. Applying the procedure explained above to random variables $\x$ in a component analysis model $\x = A \s$ turns the equation into $\x_w = \tilde{A} \s$, where $\tilde{A}$ is now an orthogonal matrix, since $I = \Cov(\x_w) = \tilde{A} \Cov(\s) \tilde{A}^\top = \tilde{A} \tilde{A}^\top$.

The orthogonal group is $\rO(n) = \{ A \in \R^{n \times n} \mid A^\top A = I\}$.  
It acts on $S^d(\R^n)$ as follows: given $A \in \rO(n)$ and $\T \in S^d(\R^n)$, define $A \bullet \T \in S^d(\R^n)$ by
\[
[A \bullet \T]_{i_1, \dots, i_d} \coloneqq \sum_{j_1, \dots, j_d = 1}^n A_{i_1j_1} \cdots A_{i_dj_d} \T_{j_1\dots j_d}.
\]
Given a linear space $V \subseteq S^d(\R^n)$, its orbit is $\rO(n) \bullet V = \{ A \bullet \T \mid A \in \rO(n), \T \in V\}$.
The following result is well-known.
\begin{lemma}\label{lem:multilinearity}
Moment and cumulant tensors are multilinear: For every $d\geq 2$ and every $A\in \R^{n\times n}$ it holds that $\mu_d(A \x) = A \bullet \mu_d(\x)$, $\kappa_d(A \x) = A \bullet \kappa_d(\x)$.    
\end{lemma}

Let $\x = A \s$, where $\s$ are (fully) independent with mean zero and unit variance, and let $A \in \rO(n)$.  
The $d$-th cumulant of $\x$ has the form 
\begin{equation}
    \label{eq:decomp}
    \kappa_d(\x) = A \bullet \kappa_d(\s) = 
 \sum_{j=1}^n \kappa_d(s_j) \ba_j^{\otimes d},
\end{equation}  where $\ba_j$ is the $j$-th column of $A \in \rO(n)$ and $\kappa_d(s_j) = \kappa_d(\s)_{j, \dots, j}$ is the $d$-th cumulant of source variable~$s_j$. This writes the cumulant as a sum of outer products of orthogonal vectors, so~$\kappa_d(\x)$ is an orthogonally decomposable (odeco) tensor. 
 
 The cumulant expression~\eqref{eq:decomp} relates usual ICA to orthogonal tensor decomposition. 
 Odeco decompositions are unique \cite[Theorem 4.1]{anandkumar2014tensor}, so the matrix is identifiable (up to sign and permutation) from $\kappa_d(\x)$ if and only if at most one $\kappa(s_j)$ is zero. The classical identifiability result for ICA \cite{comon1994independent} says that identifiability holds if and only if at most one source is Gaussian. This relates to the tensor decomposition as follows: the Gaussian distribution has zero $d$-th cumulants for all $d \geq 3$ and it is the only probability distribution with the property that there exists $d_0$ such that the $d$-th cumulant vanishes for all $d \geq d_0$ \cite{marcinkiewicz1939propriete}.
Hence, for non-Gaussian sources, we can carry out tensor decomposition of higher-order cumulants to recover $A$, see~\cite[Section 2.1]{carreno2024linear}.

We have seen that orthogonal decompositions of symmetric tensors play a role in independent component analysis. In ICA, the decompositions have core tensors in $V_{\diag}$. In this section, we study decompositions with core tensors in $V_{\pmi}$ as defined in \eqref{eq:Vpmi}. Given $\T \in S^d(\R^n)$, a unit norm vector $v \in \R^n$ is an \emph{eigenvector} of $\T$ with \emph{eigenvalue} $\lambda \in \R$ if
$\T(\cdot, v, \dots, v) = \lambda v$, viewing $\T$ as a multilinear map. We build upon the following result.

\begin{theorem}[{\cite[Theorem 1.1]{ribot2025orthogonal}}]\label{thm:main-symmetric}
    If $\T\in S^d(\R^n)$ is a generic symmetric tensor with an orthogonal basis of eigenvectors, then this basis is unique (up to sign flip).
\end{theorem}

This result is key to our discussion for the following reason. Let $\rSP(n) \subset \rO(n)$ denote the set of $n \times n$ signed permutation matrices, i.e., matrices of the form $DP$ where $D$ is diagonal with diagonal entries $\pm1$ and $P$ is a permutation matrix. The set of symmetric tensors in $S^d(\R^n)$ with an orthogonal basis of eigenvectors is the orbit $\rO(n) \bullet V_{\pmi}$. This fact, together with \Cref{thm:main-symmetric}, implies that if $\T \in V_{\pmi} \subset S^d(\R^n)$ is generic, then $Q \bullet \T \in V_{\pmi}$ if and only if $Q \in \rSP(n)$ (see Propositions 4.2 and 6.1 in \cite{ribot2025orthogonal} for details).

The following result implies that general independent distributions are sufficiently general as pairwise mean independent distributions, which we use later to prove \Cref{thm:PMICA}.

\begin{lemma}\label{lem:general-odeco}
    A generic symmetric odeco tensor in $S^d(\R^n)$ has a unique orthogonal basis of eigenvectors (up to sign flip).
\end{lemma}

\begin{proof}
    Consider an odeco tensor $\T = \sum_{j=1}^n \lambda_j q_j^{\otimes d} \in S^d(\R^n)$ where $\lambda_1, \dots, \lambda_n \neq 0$ and $\{q_1, \dots, q_n\}$ is an orthonormal basis of $\R^n$.  Then $\T$ has $\frac{(d-1)^n - 1}{d-2}$ different eigenvectors (up to scaling) in $\C^n$ given as follows: for any $1 \leq k \leq n$, any $\mathcal{J} = \{j_1, \dots, j_k\} \subseteq [n]$, and any $(k-1)$-tuple $\eta_1, \dots, \eta_{k-1}$ of $(d-2)$-nd roots of unity, there is one eigenvector $v \in \C^n$ whose coordinates with respect to the basis $\{q_1, \dots, q_n\}$ are
    \[v_j  \coloneqq \langle v, q_j \rangle= 
    \begin{cases}
        \eta_l \lambda_{j_l}^{-\frac{1}{d-2}} & \text{ if } j = j_l \text{ for some } l \in \{1, \dots, k-1\} \\
        \lambda_{j_k}^{-\frac{1}{d-2}} & \text{ if } j = j_k \\
        0 & \text{ if } j \notin \cJ
    \end{cases}
    \]
    (see \cite[Theorem~2.3]{robeva2016orthogonal}).
    Consider a collection $\cJ_1, \dots, \cJ_n $ of non-empty subsets of $[n]$, they are pairwise disjoint if and only if $\{ \cJ_1, \dots, \cJ_n\} = \{ \{1\}, \dots, \{n\}\}$, which corresponds to the set of eigenvectors $\{q_1, \dots, q_n\}$. Suppose that $\{ \cJ_1, \dots, \cJ_n\} \neq \{ \{1\}, \dots, \{n\}\}$ and, without loss of generality, suppose that $\cJ_1 \cap \cJ_2 \neq \varnothing$.
    For any $(k-1)$-tuple of $(d-2)$-nd roots of unity, imposing orthogonality between the eigenvectors corresponding to $\cJ_1$ and $\cJ_2$ leads to a non-zero polynomial expression with complex coefficients in the parameters $\left\{\lambda_j^{-\frac{1}{d-2}} \mid j \in \cJ_1 \cap \cJ_2 \right\}$. This polynomial is non-zero for generic $\lambda_1, \dots, \lambda_n$.
\end{proof}

\begin{remark}
    The previous result shows that the orthogonal basis of eigenvectors of a general odeco tensor is unique among all complex eigenvectors. The genericity of an odeco tensor $\T = \sum_{i=1}^n \lambda_i q_{i}^{\otimes d} \in S^d(\R^n)$ depends only on the scalars $\lambda_1, \dots, \lambda_n$, not on the orthonormal basis $\{q_1, \dots, q_n \}$ (see \Cref{thm:genericity-Vdiag}).
\end{remark}

\begin{proof}[Proof of Theorem~\ref{thm:PMICA}]
    The random vector $\s$ is mean independent, so it has uncorrelated entries. Rescaling its entries does not change membership of $\kappa_d(\s) \in V_{\pmi}$, since $V_{\pmi}$ is given by zero restrictions. Therefore, we can assume that $\Cov(\s) = I$ without loss of generality.
    We can also assume that $A \in \rO(n)$, after whitening $\x$.    We show that $A$ can be recovered from the $d$-th order cumulant tensor $\kappa_d(\x)$ for some $d \geq 3$ provided $\s$ is sufficiently general. 
    The idea is to apply \Cref{thm:main-symmetric} to $\kappa_d(\x)$: from \Cref{thm:V_pmi} we know that $\kappa_d(\x) \in \rO(n) \bullet V_{\pmi}$, and using \Cref{lem:general-odeco} we can ensure that $\kappa_d(\x)$ is sufficiently general.
    Consider a random vector $\s$ with $n$ pairwise mean independent entries.
    The $d$-th order cumulant tensor of $\s$ lies in $V_\pmi$, by \Cref{thm:V_pmi}. Consider a random vector $\s^{(0)}$ with $n$ independent entries that is independent from $\s$ and such that the corresponding $d$-th cumulant (odeco) tensor $\kappa_d(\s^{(0)}) = \sum_{i=1}^n \kappa_d(s^{(0)}_i) e_{i}^{\otimes d}$ is generic in the sense of \Cref{lem:general-odeco}. For each $\alpha \in \R$, let $\s^{(\alpha)} = \alpha \s + (1-\alpha)\s^{(0)}$.
    Since $\s$ and $\s^{(0)}$ are independent and by multilinearity of cumulants, we have $\kappa_d(\s^{(\alpha)})=\alpha^d \kappa_d(\s)+(1-\alpha)^d \kappa_d(\s^{(0)})$ and so $\kappa_d(\s^{(\alpha)}) \in V_{\pmi}^d$.
    Therefore, for a generic $\alpha \in \R$, $\kappa_d(\s^{(\alpha)})$ is a sufficiently generic tensor of $V_{\pmi}^d$. Hence, if $\s = \s^{(1)}$ is a sufficiently general distribution, the mixing matrix $A$ can be recovered from $\kappa_d(\x)$ up to right-multiplication by $\rSP(n)$, by \Cref{thm:main-symmetric}.
\end{proof}

\Cref{thm:PMICA} pertains to sufficiently general distributions. Our proof requires that, for some $d \geq 3$, the $d$-th order cumulant of $\s$ is generic in $V_{\pmi}$, meaning it has a unique orthogonal basis of eigenvectors. This genericity condition may be relaxed if we consider multiple higher-order cumulants. Characterizing sufficiently general PMI distributions assuming access to all higher-order cumulants is a direction for future work. In the next section, we discuss genericity conditions fixing the order of the cumulant.

A relevant comparison is ICA. There, a sufficiently general distribution is one with at most one Gaussian source. However, non-Gaussianity is not sufficient to recover the sources from a fixed $d$-th order cumulant: one also needs the $d$-th cumulant of each non-Gaussian source to be nonzero. In principle, one may need many cumulants to recover the sources. However, non-Gaussianity is not sufficient for PMICA. For example, if $\z$ is spherical (i.e. $\z$ and $Q\z$ have the same distribution for all orthogonal matrices $Q$) then $\z$ is pairwise mean independent but the cumulants contain no information about the rotation $Q$. Being non-spherical is also not sufficient: rotating the uniform distribution on $[-1,1]^2$ by an angle of~$\pi/4$ gives a PMI distribution.

\section{Sufficiently general moments and cumulants} \label{sec:proof-uniqueness-symmetric}

So far we have studied identifiability for sufficiently general tensors or distributions. Now we investigate what it means to be sufficiently general.

\begin{theorem} \label{thm:genericity-Vpmi}
    Let $\T \in V_{\pmi} \subset S^d(\R^n)$. Then $\T$ has a unique orthogonal basis of eigenvectors if and only if
    \begin{itemize}
        \item If $d=2$, its diagonal entries are distinct.
        \item If $d=3$, at most one of its diagonal entries is zero.
        \item If $d = 4$, its diagonal entries are distinct.
    \end{itemize}
\end{theorem}

\begin{proof}
For $d = 2$, the statement follows from the spectral theorem. Let $d \geq 3$, we study the binary case first, i.e., let $n = 2$ and $V_{\pmi} \subset S^d(\R^2)$. We denote the canonical basis of $\R^2$ by $\{e_0, e_1\}$. This way, the coordinates of tensors in $S^d(\R^2)$ are specified by binary strings of length $d$. For a symmetric tensor $\T \in S^d(\R^n)$, its entry $\T_{\bi}$ is determined by $|\bi| = \sum_{k=1}^d i_k$. Therefore, a symmetric tensor in $S^d(\R^2)$ is specified by $d+1$ parameters $t_0, \dots, t_d$ where $\T_{\bi} = t_{|\bi|}$ for all $\bi \in \{0,1\}^d$. This way, we have
\[
V_{\pmi} = \{ \T \in S^{d}(\R^2) \mid t_1 = t_{d-1} =0\}
\]
and its orthogonal complement is
\[
V_{\pmi}^\perp = \{ \T \in S^{d}(\R^2) \mid t_0 = t_{2} = \cdots = t_{d-2} = t_d = 0\}.
\]
Given $Q \in \rO(2)$, the linear map $Q \bullet : S^d(\R^2) \to S^d(\R^2), \T \mapsto Q \bullet \T$ can be represented by a $(d+1) \times (d+1)$ matrix indexed by $\{0, 1, \dots, d\}^2$. We define $M_Q$ as the $2 \times (d-1)$ submatrix whose rows are indexed by $\{1,d-1\}$ and columns are indexed by $\{0, 2, \dots, d-2, d\}$. The $k$-th column of this matrix is
\[
[M_Q]_k = \begin{pmatrix}
    \binom{d-1}{k} q_{00}^{d-k-1} q_{01}^k q_{10} + \binom{d-1}{k-1} q_{00}^{d-k} q_{01}^{k-1} q_{11} \\[0.5em] \binom{d-1}{k} q_{00} q_{10}^{d-k-1} q_{11}^k + \binom{d-1}{k-1} q_{01} q_{10}^{d-k} q_{11}^{k-1} 
\end{pmatrix}.
\]
Given $\T \in V_{\pmi} \subset S^d(\R^2)$, we have $Q \bullet \T \in V_{\pmi}$ if and only if $(t_0, t_2, \dots, t_{d-2}, t_d) \in \ker M_Q$.
Note that $M_{Q} = 0$ if and only if $Q \in \rSP(2)$. We are interested in $Q \notin \rSP(2)$. Let us assume, without loss of generality, that $Q \in \rSO(2)$. That is,
\[
Q = \begin{pmatrix}
    a & -b \\ b & a
\end{pmatrix} \in \rSO(2),
\]
with $a^2 + b^2 = 1$ and $a, b, \neq 0$. Indeed, if $Q \in \T \in V_{\pmi}$ for some $Q \in \rO(2) \setminus \rSO(2)$ and $\tilde{Q} \in \rSO(2)$ is obtained by flipping a column from $Q$, then $\tilde{Q} \bullet \T \in V_{\pmi}$.
If $d=3$, then
\[
M_Q = ab\begin{pmatrix}
    a & b \\
    b & -a
\end{pmatrix}
\implies  \rank(M_Q)= 2 \text{ for all }a, b \neq 0.
\]
Therefore, $\T \in V_{\pmi} \subset S^3(\R^2)$ has a unique orthogonal basis of eigenvectors if and only if~$(t_0, t_3) \neq (0,0)$, i.e., $\T \neq 0$. If $d = 4$, then
\[
M_Q = ab \begin{pmatrix}
    a^2 & -3(a^2 - b^2) & -b^2 \\
    b^2 & 3(a^2 - b^2) & -a^2
\end{pmatrix}
\implies 
\ker(M_Q) = \begin{cases}
    \langle (3,1,3)\rangle & \text{if } a \neq \pm b \\
    \langle (1,0,1), (0,1,0)\rangle & \text{if } a = \pm b.
\end{cases}
\]
So $\T \in V_{\pmi} \subset S^4(\R^2)$ has a unique orthogonal basis of eigenvectors if and only if~$t_0 \neq t_4$.

Next, we reduce the case of $n > 2$ to the binary case. Given $Q \in \rO(n)$, its normal form is a decomposition $Q = P^\top R P$, where $P \in \rO(n)$ and $R$ is block diagonal with $2 \times 2$ blocks in $\rO(2)$; see, e.g., \cite[Theorem 10.19]{roman2005advanced}.  We have $\T \in V_{\pmi}$ and $Q \bullet \T \in V_{\pmi}$ if and only if $\cS \coloneqq P \bullet \T \in P \bullet V_{\pmi}$ and $R \bullet \cS \in P \bullet V_{\pmi}$. Moreover, the coordinates of $\cS$ and $R \bullet \cS$ in the basis $\{ p_{i_1} \otimes \cdots \otimes p_{i_d} \mid i_k \in [n]\}$ are $\T_{i_1, \dots, i_d}$ and $[Q \bullet \T]_{i_1, \dots, i_d}$, respectively, where $p_j$ is the $j$-th column of $P$; see \cite[Lemma 5.12]{ribot2025orthogonal}. This reduces the problem to the binary case studied above. Therefore, $\T \in V_{\pmi} \subset S^3(\R^n)$ has a unique basis of orthogonal eigenvectors if and only if every pair of its diagonal entries are not simultaneously zero, i.e., it has at most one zero diagonal entry. For $d = 4$,  $\T \in V_{\pmi} \subset S^4(\R^n)$ has a unique basis of orthogonal eigenvectors if and only if every pair of its diagonal entries are distinct.
\end{proof}

\begin{remark}
    For PMICA, \Cref{thm:genericity-Vpmi} implies that PMI sources must follow different distributions if one wishes to identify them from the fourth-order cumulant/moment tensors.
\end{remark}

In \Cref{thm:genericity-Vpmi} we focus on the cases $d \leq 4$ because these order moments/cumulants are most commonly studied in ICA and its related methods. However, the same proof can be extended to find the genericity conditions for higher $d$. We include the next few cases in the following remark. It is an open problem to resolve the genericity conditions for all~$d$. Following the trend below, we might expect to obtain $\binom{n}{2}$ irreducible polynomials of degree~$d-3$ for $d \geq 7$.

\begin{remark} \label{remark:genericity-large-d}
    \Cref{fig:genericity_binary} shows the genericity conditions in $V_{\pmi} \subset S^d(\R^2)$ for $2 \leq d \leq 9$. Those pictures are of the three-dimensional slice obtained by setting $t_1 = 0$ and $t_k = k$ for $1 < k < d-1$, following the notation introduced in the proof above. The cases where $d \leq 4$ were addressed in \Cref{thm:genericity-Vpmi}. For $d \geq 5$, having distinct or nonzero diagonal entries is not sufficient for $\T \in V_{\pmi} \subset S^d(\R^n)$ to have a unique basis of orthogonal eigenvectors, since the kernel of the matrix $M_Q$ studied in the proof of \Cref{thm:genericity-Vpmi} is more complicated. For example, for $d = 5$ we get
    \[
    \ker(M_Q) = \langle (4(a^2 - b^2), a^2 - b^2, -ab, 2ab), (-2ab, ab, a^2-b^2, 4(a^2 - b^2)) \rangle,
    \]
    so $\T \in V_{\pmi} \subset S^5(\R^2)$ has a unique basis of orthogonal eigenvectors if and only if 
    \[
    t_{0}^{2}-2\,t_{0}t_{2}-8\,t_{2}^{2}-8\,t_{3}^{2
      }-2\,t_{3}t_{5}+t_{5}^{2} \neq 0.
    \]
    For $d = 6$, we get that $\ker(M_Q)$ is
    \begin{align*}
    \begin{cases}
        \langle (5,1,0,1,5), (5(a^2 - b^2), a^2 - b^2, -ab, 0, 0), (-10ab, 0, a^2-b^2, -2ab, 0)\rangle & \text{if } a \neq \pm b \\
        \langle (-5,1,0,0,0), (0,0,1,0,0), (5,0,0,1,0), (1,0,0,0,1)\rangle & \text{if } a = \pm b
    \end{cases}
    \end{align*}
    so $\T \in V_{\pmi} \subset S^6(\R^2)$ has a unique basis of orthogonal eigenvectors if and only if
    \[
    t_{0}-5\,t_{2}-5\,t_{4}+t_{6} \neq 0 \quad \text{and} \quad t_{0}+5\,t_{2}-5\,t_{4}-t_{6} \neq 0.
    \]
    For $d=7$, the genericity conditions in $V_{\pmi}$ are given by a quartic with 63 monomials. For~$d=8$, they are given by a quintic with 204 monomials. For $d=9$, they are given by a sextic with 752 monomials\footnote{We did these computations using \texttt{Macaulay2} \cite{M2}.}.
    Using the same argument as above, the genericity conditions in~$S^d(\R^n)$ can be obtained by reducing the problem to the binary case and arguing about pairs of coordinates. For example, for $d=5$ we obtain that $\T \in V_{\pmi} \subset S^5(\R^n)$ has a unique basis of orthogonal eigenvectors if and only if
    \[
    \T_{iiiii}^2 - 2\T_{iiiii}\T_{iiijj} - 8\T_{iiijj}^2 - 8\T_{iijjj}^2 - 2\T_{iijjj}\T_{jjjjj} + \T_{jjjjj}^2 \neq 0 \quad \text{for all } i \neq j.
    \]
\end{remark}

Recall that $V_{\diag} \subseteq V_{\pmi}$ and that the odeco tensors are the ones that lie in $V_{\diag}$ after an orthogonal change of basis. In \Cref{lem:general-odeco} we showed that a generic symmetric odeco tensor has a unique orthogonal basis of eigenvectors. The following result specifies what generic means in this context.

\begin{theorem} \label{thm:genericity-Vdiag}
    Let $\T \in V_{\diag} \subset S^d(\R^n)$.
    \begin{itemize}
        \item If $d$ is even, $\T$ has a unique basis of orthogonal eigenvectors if and only if its diagonal entries are distinct.
        \item If $d$ is odd, $\T$ has a unique basis of orthogonal eigenvectors if and only if at most one of its diagonal entries is zero.
    \end{itemize}
\end{theorem}

\begin{proof}
    We use the same proof idea as in \Cref{thm:genericity-Vpmi}. First, let $n =2$.
    Let $\tilde{M}_Q$ be the $2 \times 2$ submatrix of the matrix $M_Q$ defined in the proof of \Cref{thm:genericity-Vpmi} corresponding to the columns indexed by $0$ and $d$. That is,
    \[
    \tilde{M}_Q = \begin{pmatrix}
        q_{00}^{d-1} q_{10} & q_{01}^{d-1}q_{11} \\
        q_{00}q_{10}^{d-1} &  q_{01} q_{11}^{d-1}
    \end{pmatrix}.
    \]
    Given $\T \in V_{\diag}$, $Q \bullet \T \in V_{\pmi}$ if and only if $(t_0, t_d) \in \ker(M_Q)$. We are interested in $Q \in \rO(2) \setminus \rSP(2)$.
    If $d$ is odd, $M_{Q}$ has full rank for all $Q \in \rO(2) \setminus \rSP(2)$. If $d$ is even, $M_{Q}$ has full rank for all $Q \in \rO(2) \setminus \rSP(2)$ except if $q_{00} = \pm q_{01}$ and, in such a case, $\ker (M_Q) = \langle (1, -1)\rangle$. This concludes the proof for $n=2$. We reduce the case $n \geq 3$ to the binary case by considering the normal form of orthogonal matrices and using the same reasoning as in the proof of \Cref{thm:genericity-Vpmi}.
\end{proof}

\begin{remark}
    In \cite{ribot2025orthogonal} we also study tensors in $\R^{n_1 \times \cdots \times n_d}$ with an orthogonal basis of singular vector tuples. In that case, $V_{\pmi}$ is replaced by an analogous linear space $V$ and the action of $\rO(n)$ on $S^d(\R^n)$ is replaced by the action of $\rO(n_1) \times \cdots \times \rO(n_d)$ on $\R^{n_1} \times \cdots \times \R^{n_d}$. For $(\R^2)^{\otimes 3}$, the set of tensors with an orthogonal basis of singular vector tuples is precisely the set of odeco tensors, i.e., $V = V_{\diag}$. Therefore, using the same reasoning as above, we can conclude that a tensor in $V \subset \R^{n_1 \times n_2 \times n_3}$ has a unique basis of orthogonal singular vector tuples if and only if at most one of its diagonal entries is zero. For $(\R^2)^{\otimes 4}$, we get that the tensors $\T \in V$ such that $(Q_1, \dots, Q_4) \cdot \T \in V$ with some $Q_k \notin \rSP(2)$ live in a four-dimensional variety consisting of $14$ linear components:
    \begin{align*}
        &\T_{0000} \pm \T_{0011} = \T_{0101} \mp \T_{0110} = \T_{1001} \mp \T_{1010} = \T_{1100} \pm \T_{1111} = 0, \\
        & \T_{0000} \pm \T_{0101} = \T_{0011} \mp \T_{0110} = \T_{1001} \mp \T_{1100} = \T_{1010} \pm \T_{1111} =0, \\
        & \T_{0000} \pm \T_{0110} = \T_{0011} \mp \T_{0101} = \T_{1010} \mp \T_{1100} = \T_{1001} \pm \T_{1111} = 0, \\
        & \T_{0000} \pm \T_{1001} = \T_{0011} \mp \T_{1010} = \T_{0101} \mp \T_{1100} = \T_{0110} \pm \T_{1111} = 0, \\
        & \T_{0000} \pm \T_{1010} = \T_{0011} \mp \T_{1001} = \T_{0110} \mp \T_{1100} = \T_{0101} \pm \T_{1111} = 0, \\
        & \T_{0000} \pm \T_{1100} = \T_{0101} \mp \T_{1001} = \T_{0110} \mp \T_{1010} = \T_{0011} \pm \T_{1111} = 0,  \\
        & \T_{0000} \pm \T_{1111} = \T_{0011} \pm \T_{1100} = \T_{0101} \pm \T_{1010} = \T_{0110} \pm \T_{1001} = 0.
    \end{align*}
\end{remark}

\section{Other relaxations of independence}\label{sec:other_relaxations}

Here, we prove \Cref{thm:general-ca}; i.e., we show generic identifiability of component analysis for any condition on the cumulants of the source variables that is weaker than independence ($V_{\diag}$) and stronger than pairwise mean independence ($V_\pmi$). Then we  study examples of such distributions. We close the section by proving \Cref{thm:pmi-maximal}; i.e., PMICA becomes generically unidentifiable if one mean independence assumption is dropped.

\begin{proof}[Proof of \Cref{thm:general-ca}]
    We use the same idea as in the proof of \Cref{thm:PMICA}. The $d$-th cumulant tensor of a sufficiently general independent distribution is generic as a tensor in $V_{\pmi}$ in the sense of \Cref{thm:main-symmetric}, by \Cref{lem:general-odeco}. We are assuming that $V_{\diag} \subseteq V$. Hence, if $\kappa_d(\s)$ is sufficiently general in $V$, it is also sufficiently general in $V_{\pmi}$. Therefore, the result follows from \Cref{thm:PMICA}.
\end{proof}
\begin{remark}
    In \Cref{thm:general-ca} we assume that $\Cov(\s) = I$ so that we can restrict our search space to $\rO(n)$ by applying the whitening transformation to $\x$. This assumption was not needed in \Cref{thm:PMICA} because PMI implies uncorrelatedness and scaling $\s$ does not change membership in the model, since $V_{\pmi}$ is given by zero restrictions. Similarly, this assumption is not needed in \Cref{thm:general-ca} whenever $V \subset S^d(\R^n)$ is given by zero restrictions.
\end{remark}

\begin{lemma} \label{lem:no-sign-condition}
    For $d \geq 3$, any real number can be obtained as the $d$-th cumulant of some probability distribution.
\end{lemma}
\begin{proof}
    Given a random variable $z$ and $\alpha \in \R$ we have $\kappa_d(\alpha z) = \alpha^d \kappa_d(z)$, so it suffices to study the sign of cumulants.
    Let $z$ be a Bernoulli random variable with parameter $p$, and denote its $d$-th cumulant as $\kappa_d(p)$. Then $\kappa_d(p)$ is a polynomial of degree $d$ in $p$ that satisfies the following recursion: $\kappa_1 = p$ and $\kappa_{d+1} = p(1-p) \frac{d}{dp}\kappa_d$. Fix $d \geq 2$, then $\kappa_d(0) = \kappa_d(1) = 0$ and $\kappa_d(p)$ is not identically zero, so there exist $p^+, p^- \in (0,1)$ such that $\frac{d}{dp}\kappa_d(p^+) > 0$ and $\frac{d}{dp}\kappa_d(p^-) < 0$, by the mean value theorem. Therefore, $\kappa_{d+1}(p^+) > 0$ and $\kappa_{d+1}(p^-) < 0$.
\end{proof}

\begin{proposition}
     Any symmetric tensor $\T \in S^d(\R^n)$ of order $d \geq 3$ can be obtained as a cumulant tensor of some probability distribution.
\end{proposition}

\begin{proof}
    There is no sign condition on univariate cumulants for $d\geq 3$ (only for $d = 2$ where the variance must be non-negative), by \Cref{lem:no-sign-condition}. Given any $\T \in S^d(\R^n)$, consider a symmetric decomposition $\T = \sum_{i=1}^r \lambda_i \mathbf{v}_i^{\otimes d}$, which always exists by \cite[Lemma 4.2]{comon2008symmetric}. 
    Construct $r$ independent random variables $z_1, \dots, z_r$ such that $\kappa_d(z_i) = \lambda_i$ and define the random vector $\y = \sum_{i=1}^r z_i \mathbf{v}_i$. Then $\kappa_d(\y) = \sum_{i=1}^r \kappa_d(z_i\mathbf{v}_i) = \T$.
\end{proof}

The previous proposition implies that it is not necessary to assume that $V_{\diag} \subseteq V$ in \Cref{thm:general-ca}, and the statement holds as long as $V$ contains sufficiently generic tensors in $V_{\pmi}$. For example, one can take $V$ to be the subset of diagonal tensors whose first diagonal entry is zero, which contains sufficiently general tensors in $V_{\pmi}$ for $d=3,4$. This recovers the classical ICA identifiability condition that one source may be Gaussian \cite{comon1994independent}.

When we talk about a tensor $\T$ being generic in $V \subset S^d(\R^n)$ in the context of component analysis, we mean that $\{Q \in \rO(n) \mid Q \bullet \T \in V\} \subseteq \rSP(n)$, following \cite{mesters2022non}.
One benefit of \Cref{thm:general-ca} is that one does not need to study the genericity conditions for each $V \subseteq V_{\pmi}$ as long as $V_{\diag} \subseteq V$. Hence, our identifiability result can be applied to any domain-specific family of distributions containing generic independent distributions (in the sense of \Cref{thm:genericity-Vdiag}) and having cumulant tensors in $V_{\pmi} \subset S^d(\R^n)$ for some $d \geq 3$. For example, this can be applied to a parametric family of distributions satisfying such properties, which can be useful in applications. We formalize this discussion as follows.

\begin{corollary}
    Let $\mathcal{F}$ be a family of probability distributions that contains generic independent distributions and whose cumulants tensors lie in $V_{\pmi} \subset S^{d}(\R^n)$ for some $d \geq 3$. Consider the model $\x = A \s$ where $A \in \R^{n\times n}$ is invertible and $\s \in \mathcal{F}$ is sufficiently general. Then $A$ is identifiable from $\kappa_d(\x)$ (up to permutation and scaling of columns).
\end{corollary}

The following is a consequence of \Cref{thm:general-ca}; see \cite[Proposition 7.1]{ribot2025orthogonal} for details.

\begin{corollary}
    Let $V_{\diag} \subseteq V \subseteq V_{\pmi}$. Then,
    \[
    \dim(\rO(n) \bullet V) = \dim(V) + \dim(\rO(n)) = \dim(V) + \binom{n}{2}.
    \]
\end{corollary}

When we recover the mixing matrix from cumulants, the ICA model corresponds to symmetric odeco tensors $\rO(n) \bullet V_{\diag}$ and the PMICA model corresponds to tensors with and orthogonal basis of eigenvectors $\rO(n) \bullet V_{\diag}$. 
Given $d \geq 3$, $\dim (\rO(n) \bullet V_{\diag}) = \binom{n+1}{2} \asymp n^2$ and 
$\dim(\rO(n) \bullet V_{\pmi}) = \binom{n+d-1}{d} - \binom{n}{2} \asymp n^d$. This shows how PMICA is more expressive than ICA: given a cumulant tensor $\kappa_d(\x)$ in practice, we expect to approximate it better  with $\rO(n) \bullet V_{\pmi}$ than with $\rO(n) \bullet V_{\diag}$.

\subsection{Examples between independence and PMI}

We give examples of conditions that are stronger than PMI but weaker than independence, so \Cref{thm:general-ca} applies.

Pairwise mean independence is stronger than uncorrelatedness: if $z_1$ is mean independent of $z_2$, then $\E(z_1z_2) = \E(\E(z_1 \mid z_2) z_2) = \E(\E(z_1) z_2) = \E(z_1)\E(z_2)$. To see that both notions are not equivalent, take the vector $(z,z^2)$ where $z$ is standard Gaussian. In this case ${\rm cov}(z,z^2)=\E (z^3)=0$ but $\E(z^2 \mid z)=z^2$, which is not equal to $\E(z^2)=1$ almost surely. 

\subsubsection{Reflectional symmetries on cumulants}
The conditions in \cite{mesters2022non} (general common covariance, multiple scaled elliptical distribution, mean independent) are special cases of our setting, since the set of reflectionally invariant tensors $V_{\mathrm{refl}}$ is included in $V_{\pmi}$. Indeed, $V_{\mathrm{refl}} =  \{ \T \in S^d(\R^n) \mid \T = D \bullet \T \text{ for all } D \in \diag(\{\pm1 \}^n)\}$, where $\diag(\{\pm1 \}^n)$ is the set of diagonal matrices with $\pm1$ on the diagonal \cite[Section 5]{mesters2022non}, so in $S^d(\R^n)$ we get
\[
\dim(V_{\mathrm{refl}}) = \begin{cases}
    0 & \text{if } d = 2k + 1\\
    \binom{n + k - 1}{k} & \text{if } d = 2k,
\end{cases}
\]
which is smaller than $\dim(V_{\mathrm{pmi}}) = \binom{n + d - 1}{d} - n(n-1)$.

\subsubsection{Mean independence} 

Define
$$
V_{\mi}\;:=\;\{\cT\in S^d(\R^n)\mid \cT_{i,j_2,\ldots,j_d} = 0 \mbox{ if }i\neq j_2,\ldots,j_d\}.
$$
In words, $\cT_{j_1,j_2,\ldots,j_d} = 0$ if there is an index that appears only once in the tuple $(j_1,j_2,\ldots,j_d)$.
\Cref{thm:V_pmi} holds replacing PMI by mean independence and $V_{\pmi}$ by $V_{\mi}$. To see this, we can follow the proof of \Cref{thm:V_pmi} replacing \eqref{eq:wedgepi0} with
\begin{equation}\label{eq:wedgepi02}
    \E(z_i z_{j_2}\cdots z_{j_d})\;=\;\E \big(z_{j_2}\cdots z_{j_d}\,\E(z_i \mid z_{\setminus i})\big)\;=\;\E(z_i)\E(z_{j_2}\cdots z_{j_d})\qquad(i\ne j_2,\ldots,j_d).
\end{equation}
In \eqref{eq:KinM} specialize $i_1,i_2,\ldots,i_d$ to $i,j_2,\ldots,j_d$ with $i\neq j_2,\ldots,j_d$ with the same special partition $\pi_0$ to conclude that with mean independence $\kappa_d(\x)\in V^{d,n}_{\rm pmi}$.

We have the containment $V_{\mi}^{d}\subseteq V_{\pmi}^{d}$ with strict inclusion for $d \geq 3$ provided $n \geq 3$, so mean independence implies PMI. This property can also be seen as follows.

\begin{proposition}
    If $\z=(z_1,\ldots,z_n)$ is mean independent, it is pairwise mean independent. 
\end{proposition}
\begin{proof}
By the tower property of conditional expectation, for any $i\neq j$,
   $\E(z_i \mid z_j) = \E(\E(z_i \mid z_{\setminus i}) \mid z_j) = \E(z_i)$ almost surely.
\end{proof}

\subsubsection{Pairwise independence} A random vector $\z$ is \emph{pairwise independent} if $z_i \independent z_j$ for all $i \neq j$. This is equivalent to $\kappa(\z)_{i_1, \dots, i_d} = 0$ whenever $|\{ i_1, \dots, i_d\}| = 2$, provided that $K_\z(\mathbf{t})$ is sufficiently smooth around zero. Let
\[
V_{2-\mathrm{indep}}^d = \{ \T \in S^d(\R^n) \mid \T_{i_1, \dots, i_d} = 0 \text{ if } |\{ i_1, \dots, i_d\}| = 2\}.
\]
Let $\cI_{2-\mathrm{indep}} = \{ \bi \in [n]^d \mid 1 \leq i_1 \leq \dots \leq i_d \leq n, |\{ i_1, \dots, i_d\}| = 2\}$. Then $\codim(V_{2-\mathrm{indep}}^d) = |\cI_{2-\mathrm{indep}}| = (d-1)\binom{n}{2}$ and $V_{2-\mathrm{indep}}^d \subseteq V_{\mathrm{PMI}}^d$ for all $d \geq 2$.

\subsubsection{$k$-wise independence}. Let $k \in [n]$. A random vector $\z$ is \emph{$k$-wise independent} if every set of $k$ distinct entries of $\z$ are jointly independent. 

\begin{remark}
    Independence $\implies$ $k$-wise independence ($k \geq 2$) $\implies$ pairwise independence. More precisely,
    $k$-wise independence $\implies$ $(k-1)$-independence.
\end{remark}

If $K_\z(\mathbf{t})$ is sufficiently smooth around zero, then $k$-wise independence is equivalent to~$\kappa(\z)_{i_1, \dots,i_d} = 0$ whenever $2 \leq |\{ i_1, \dots, i_d\}| \leq k$. Let 
\[
V_{k-\mathrm{indep}}^d = \{ \T \in S^d(\R^n) \mid \T_{i_1, \dots, i_d} = 0 \text{ if } 2 \leq |\{ i_1, \dots, i_d\}| \leq k\}
\]
Let
$\cI_{k-\mathrm{indep}} = \{ \bi \in [n]^d \mid 1 \leq i_1 \leq \dots \leq i_d \leq n, 2 \leq |\{ i_1, \dots, i_d\}| \leq k\}$. Then $\codim(V_{k-\mathrm{indep}}^d) = |\cI_{k-\mathrm{indep}}| = \sum_{j=2}^k \binom{d-1}{j-1} \binom{n}{j}$ and $V_{k-\mathrm{indep}}^d \subseteq V_{\mathrm{PMI}}^d$ for all $d \geq 2$.

\begin{remark}
    We have $V_{k-\mathrm{indep}}^d = V_{(k+1)-\mathrm{indep}}^d \subset S^d(\R^n)$ whenever $k \geq \min\{d, n\}$. That is, the $d$-th order cumulant does not distinguish between $k$-wise independence and independence whenever $k \geq \min\{d, n\}$.
\end{remark}

\subsubsection{Correlation of energies}\label{sec:correlation-energies} Let $z_i = \sigma_i \varepsilon_i$, where $\sigma_i,\varepsilon_i$ are random variables, $\varepsilon_i \independent \varepsilon_j$ for all $i \neq j$ and $\sigma_i \independent \varepsilon_j$ for all $i,j$. That is, $z_i \independent z_j \mid \sigma_i, \sigma_j$ but $z_i \not\independent z_j$. Assume that $\sigma_i > 0 \text{ a.s}$, $\E (\varepsilon_i) = 0$ and $\E (\varepsilon_i^2) = 1$. Hence, $\E(z_i) = 0$. Topographic ICA (tICA) \cite{hyvarinen2001topographic} is a particular case of this model.

\begin{proposition}
    Correlation of energies implies pairwise mean independence.
\end{proposition}
\begin{proof}
    $
    \E(z_i \mid z_j) = \E(\sigma_i \varepsilon_i \mid \sigma_j \varepsilon_j) = \E(\sigma_i \mid \sigma_j \varepsilon_j) \E(\varepsilon_i) = 0 = \E(z_i).
    $
\end{proof}

\begin{example}[Broadcasting on trees]
Let $T=(V,E)$ be a rooted tree with root denoted by $0$ and non-root leaves denoted by $\{1, \dots, n\}$.  
For a vertex $v\in V$, we write $u \preceq v$ if $u$ lies on the unique path from the root $0$ to $v$. Let $\{\tau_v \mid v \in V\}$ be independent random variables with zero mean. Consider the broadcasting process $\{z_v: v\in V\}$ defined as: $z_v = \prod_{u \preceq v} \tau_u$. Then $\z = (z_1, \dots, z_n)$ is pairwise mean independent.
\end{example}

\subsubsection{Spherical distributions} In spherical distributions the components are mean independent and pairwise mean independent but not independent unless $Z$ is Gaussian; see \cite{kelker1970distribution,rossell2021dependence}.  However, cumulants of spherical distributions are insufficiently general in~$V_\pmi$. This is because their distribution is preserved under rotation.

\subsection{Generic non-identifiability beyond pairwise mean independence}

We prove \Cref{thm:pmi-maximal}. Being mean independent is not symmetric: $\E(z_1 \mid z_2) = \E(z_1)$ does not imply that $\E(z_2 \mid z_1) = E(z_2)$. For example, take $z_1 = U(-1,1)$ and $z_2 = z_1^2$. Then $\E(z_1 \mid z_2) = 0 = \E(z_1)$ but $\E(z_2 \mid z_1) = z_1^2 \neq 1/3 = \E(z_2)$ almost surely. If $z_2$ is not mean independent of $z_1$ we cannot guarantee that $\kappa(\x)_{21\dots 1}$ is zero. The following shows that dropping one mean independence condition leads to generic unidentifiability of the component analysis model. That is, if $s_1$ is mean independent of $s_2$ but $s_2$ is not mean independent of $s_1$, then there are many rotations that preserve this property.

\begin{lemma}\label{lem:Vpmi-maximal}
    Let $d \geq 3$ and consider a linear space $V_{\pmi} \subsetneq V \subseteq S^{d}(\R^n)$ given by zero restrictions. Let $\T \in V$ be generic. Then, there exists $Q \in \rO(n) \setminus \rSP(n)$ such that $Q \bullet \T \in V$.
\end{lemma}

\begin{proof}
    Consider first $n=2$. Then $V_{\pmi} = \{ \T \in S^d(\R^2) \mid \T_{12\ldots2} = \T_{21\ldots1} = 0\}$. Suppose that $V = \{ \T \in S^d(\R^2) \mid \T_{12\ldots2} = 0\}$, the argument for the other cases is analogous. Then $V$ is the set of symmetric tensors with $e_2$ as an eigenvector, so $\rO(2) \bullet V = S^d(\R^2)$ because every symmetric tensor has an eigenvector. A generic tensor in $S^d(\R^2)$ has $\frac{(d-1)^2 - 1}{d-1}$ distinct complex eigenvectors \cite[Theorem 5.5]{cartwright2013number}. At least two of these eigenvectors are real, corresponding to the maximizer and minimizer of $\T(x, \dots, x)$ subject to $\|x\|=1$. Moreover, these two real eigenvectors are not orthogonal to each other by genericity of $\T$ and \cite[Proposition 7.1.]{ribot2025orthogonal}.
    This means that a generic fiber of the parametrization $\phi : \rO(2) \times V \to S^d(\R^n), (Q, \T) \mapsto Q \bullet \T$ is finite, but it is not included in $\rSP(2) \times V$. Hence, for a generic $\T \in V$, there exists $Q \in \rO(2) \setminus \rSP(2)$ such that $Q \bullet \T \in V$.

    Next, let $n \geq 3$ and let $\cI_{\pmi} = \{ (i, j, \dots, j) \in [n^d] \mid i \neq j\}$. Then $V_{\pmi} = \{\T \in S^{d}(\R^n) \mid \T_{\bi} = 0 \text{ for all } \bi \in \cI_{\pmi}\}$. Let $\cI = \cI_{\pmi} \setminus \{(2,1,\dots,1)\}$ and let $V = \{\T \in S^{d}(\R^n) \mid \T_{\bi} = 0 \text{ for all } \bi \in \cI\}$. By the binary case, given a generic $\T \in V$, there exists a matrix 
    \[
    Q = \begin{pmatrix}
        \tilde{Q} & & &\\
        & 1 & & \\
        & & \ddots &\\
        & & & 1
    \end{pmatrix}
    \]
    with $\tilde{Q} \in \rO(2)\setminus \rSP(2)$ such that $Q \bullet \T \in V$.
\end{proof}

\begin{proof}[Proof of \Cref{thm:pmi-maximal}]
Without loss of generality, suppose $s_2$ is not mean independent of $s_1$. Here $\s$ being general means that for all $d \geq 3$, $\kappa_d(\s)$ is generic in $V = \{ \T \in S^{d}(\R^n) \mid \T_{ij\dots j} = 0 \text{ for all } i \neq j, (i,j) \neq (2,1)\}$, so the statement follows by \Cref{lem:Vpmi-maximal}.
\end{proof}

\Cref{thm:pmi-maximal} says that one cannot identify $A$ from $\kappa_d(\x)$ if
if one does not imposes all the zero-restrictions $\kappa_d(\s)$ coming from mean independence. Generic in this context means that~$\kappa_d(\s) \not \in V_{\pmi}$ for any $d \geq 3$. However, \Cref{thm:general-ca} says that it is enough to have a generic $\kappa_d(\s) \in V_{\pmi}$ for only one $d \geq 3$ to have identifiability of the model, even if $\s$ is not PMI. The following is an example of such a distribution.

\begin{example}
    Let $\s = (s_1, s_2)$ with $s_1 \sim \mathcal{N}(0,1)$ and $s_2 = (s_1^2 - 1)/\sqrt{2}$. We have $\E(s_1) = \E(s_2) = 0$ and $\E(s_1 \mid s_2) = 0$ but $\E(s_2 \mid s_1) = (s_1^2 - 1)/\sqrt{2} \neq 0$ almost surely, so $\s$ is not PMI. However, $\kappa_4(\s)_{1222} = \kappa_4(\s)_{2111} = 0$, so $\kappa_4(\s) \in V_{\pmi}$. Moreover, $\kappa_4(\s)$ is sufficiently general in $V_{\pmi}$ in the context of \Cref{thm:general-ca} because $\kappa_4(\s)_{1111} = 0 \neq \kappa_4(\s)_{2222} = 12$; see \Cref{thm:genericity-Vpmi}. Therefore, in the model $\x = A \s$, $A$ can be identified from $\kappa_4(\x)$, up to permutation and sign flip of columns.
\end{example}

\section{A minimum-distance estimator:
         consistency and finite-sample behavior}
\label{sec:estimator}

Our identifiability analysis applies to any model obtained by restricting the cumulant at some fixed $d \geq 3$ to lie in a linear subspace $V \subset S^d(\R^n)$ with 
\begin{equation}\label{eq:sandwich}
V_{\diag}^{d,n}\ \subseteq\ V\ \subseteq\ V_{\pmi}^{d,n}.    
\end{equation}
Our two main examples are ICA and PMICA.
We observe $\x=A\s$, where $\s$ is centered with $\Cov(\s)=I$, and we assume that $\kappa_d(\s)$ lies in $V$ and is generic in the sense of Section~\ref{sec:proof-uniqueness-symmetric}.

\subsection{Minimum distance estimator}

Consider i.i.d.\ samples $\x_1,\dots,\x_N\in\R^n$ from $\x$. Stack the samples to form the matrix $X=[\,\x_1,\dots,\x_N\,]^\top\in\R^{N\times n}$ and whiten them to obtain $X_w$. Then $\x_w = \tilde{A} \s$ with $\tilde{A} \in \rO(n)$,
as explained in Section~\ref{sec:pmica}.

For $d\ge3$, let $\widehat h_d(\x_w)$ denote an order-$d$ estimator, for example the sample moment tensor $\widehat\mu_d(\x_w)$, sample cumulant tensor $\widehat\kappa_d(\x_w)$, or an alternative cumulant estimator such as the order-$d$ $k$-statistics; see \cite[Chapter~4]{mccullagh2018tensor}. Let $h_d(\x_w)$ be the population version.
\medskip

\noindent\textbf{Population formulation.}
Let $\Pi_W$ be the orthogonal projector (with the Frobenius inner product) onto a subspace $W\subset S^d(\R^n)$. For a fixed $V$ satisfying \eqref{eq:sandwich} define
\begin{equation*}
g(Q)\ :=\ \Pi_{V^\perp}\!\bigl(Q^{\top}\bullet h_d(\x_w)\bigr),\qquad Q\in\rO(n).
\end{equation*}
Using multilinearity of moments and cumulants (Lemma~\ref{lem:multilinearity}),
\[
g(Q)=\Pi_{V^\perp}\!\bigl((Q^\top\tilde A)\bullet h_d(\s)\bigr).
\]
The squared Euclidean distance between the tensor $Q^{\top}\bullet h_d(\x_w)$ and the linear space $V$ is found by the optimization problem
\begin{equation}\label{eq:FU}
    \mbox{minimize}\quad F(Q):=\|g(Q)\|_{\rm F}^2,\qquad Q\in\rO(n).
\end{equation}
The following lemma follows from \Cref{thm:general-ca}.

\begin{lemma}[Population identification]\label{lem:pop-ident}
Suppose $h_d(\s)$ is generic in $V$. Then $g(Q)=0$ if and only if $Q^\top\tilde A\in{\rm SP}(n)$. Hence, the minimizers of $F$ are $\tilde AP$ with $P\in{\rm SP}(n)$.
\end{lemma}

Letting $\Sigma=\Cov(\x)$, recall from Section~\ref{sec:pmica} that $\tilde A=\Sigma^{-1/2}A$ and so we recover $A$ up to signed permutation via
\begin{equation}\label{eq:A0fromtilde}
    \Sigma^{1/2}\,\tilde AP\;=\;AP,\qquad P\in{\rm SP}(n).
\end{equation}

\begin{remark}[Equivalent viewpoints]
Lemma~\ref{lem:pop-ident} shows that the minimizers $Q$ of $F$ are signed permutations of $\tilde A$.  
Equivalently, the columns of $Q^\top\tilde A$ are an orthonormal set of eigenvectors of the tensor $h_d(\s)$.  
From a variational perspective, the eigenvectors are the stationary points of
\[
u\ \longmapsto\ \langle h_d(\s),\,u^{\otimes d}\rangle
\]
on the unit sphere.  
See \cite[Sec.~4]{ribot2025orthogonal} for a discussion of this eigenvector formulation.
\end{remark}

\noindent\textbf{Sample formulation.}
Mimicking the population construction, we estimate $\tilde A$ by solving
\begin{equation}\label{eq:Ustar}
 \widehat Q\ \in\ \argmin_{Q\in\rO(n)}\ F_N(Q):=\|g_N(Q)\|_{\rm F}^2,
 \qquad
 g_N(Q):=\Pi_{V^\perp}\!\bigl(Q^{\top}\bullet \widehat h_d(\x_w)\bigr).
\end{equation}
We solve this optimization problem 
using Riemannian gradient descent (RGD) on $\rO(n)$. Analogously to \eqref{eq:A0fromtilde}, to estimate $A$, compute the sample covariance $\widehat \Sigma$ of the data and set
\[
\widehat A_N\;=\;\widehat{\Sigma}^{1/2}\,\widehat Q.
\]

\subsection{Large-sample theory}

The criterion $F_N(Q)=\|\Pi_{V^\perp}(Q^\top\bullet\widehat h_d)\|_{\rm F}^2$ leads to the Generalized Method of Moments (GMM) estimator, with moment conditions the defining equations of $V$.  The standard GMM results apply; we state them for completeness and refer to \cite[Propositions~6.2--6.3]{mesters2022non} and \cite{hansen1982gmm} for proofs, generalizations, and discussion.

\begin{proposition}[Consistency]\label{prop:consistency}
Let $\x_{1},\dots,\x_{N}$ be i.i.d. samples from $\x=A\s$. Assume that $\kappa_{2}(\s)=I_{n}$, $\kappa_{d}(\s)$ is generic in $V$, and $\E\|\x\|^{d}<\infty$. Then $\widehat A_N\to_p AP$ for some $P\in{\rm SP}(n)$.
\end{proposition}

Let ${\rm vec}g:\R^{n^2}\to \R^{\dim(V^\perp)}$ denote the map $g$ after vectorizing its domain and codomain. Define ${\rm vec}g_N$ similarly. Let $G(A)\in \R^{n^2\times \dim(V^\perp)}$ denote the Jacobian of ${\rm vec}g$ at $A$ and define $\Xi=\lim_{N\to \infty} \Var(\sqrt{N}\,{\rm vec}g_N(AP))$.

\begin{proposition}[Asymptotic normality]
Under the assumptions of Proposition~\ref{prop:consistency}, assume additionally that $\E\|\x\|^{2d}<\infty$. Write $G:=G(AP)$. Then
\[
\sqrt{N}\,\operatorname{vec}(\widehat A_N-AP)\ \to \mathcal{N}\!\Bigl(0,\ (G^\top G)^{-1}G^\top \,\Xi\, G\,(G^\top  G)^{-1}\Bigr).
\]
\end{proposition}
This shows that our estimator has the standard good asymptotic properties: consistency and asymptotic normality. Next, we provide finite sample bounds for sub-Gaussian data.

\subsection{Finite-sample error bounds under sub-Gaussianity}\label{sec:finitebounds}

We develop basic finite sample analysis. We focus on moments to keep things simple; c.f. Corollary~\ref{cor:momentssuff}. The proof of \Cref{prop:param-error} and technical material used in this section is in \Cref{sec:appendix-finite-bounds}.

\begin{assumption}[Sub-Gaussian data]\label{asmp:sg-4} The data $\x_1,\ldots,\x_N$ are whitened, so $\x\in\mathbb{R}^n$ satisfies $\E[\x]=0$ and $\E[\x\x^\top]=I$, and $\x_1,\ldots,\x_N$ are i.i.d.\ copies of $\x$. We also assume $\x$ is a  sub-Gaussian vector, that is, there exists $\sigma>0$ such that $\langle u,\x\rangle$ is $\sigma$-sub-Gaussian for every unit vector $u \in \R^n$.
\end{assumption}

Define the $d$-th order sample moment tensor $\widehat\mu_d \;=\; \tfrac{1}{N}\sum_{r=1}^N \x_r^{\otimes d}$ and the population moment tensor  
$\mu_d \;=\; \E[\x^{\otimes d}]$. We bound the spectral norm (see \cite{lim2005singular})
\begin{equation}\label{eq:spnorm}
\|\widehat\mu_d-\mu_d\|_2\;:=\;\max_{\|u\|=1} \<\widehat\mu_d-\mu_d,u^{\otimes d}\>.   
\end{equation}
For any
unit vector $u\in\R^n$, define $\z=u^\top \x$ and $\z_i=u^\top \x_i$ for $i=1,\ldots,N$. Then
\[
\langle \widehat\mu_d- \mu_d, u^{\otimes d}\rangle \;=\; \frac{1}{N}\sum_{i=1}^N \left(z_i^d-\E\![\z^d]\right)\;=\;\frac1N\sum_{i=1}^N \left(\z_i^d-\E [\z_i^d]\right).
\]
This allows the use of concentration of measure techniques to bound $\|\widehat\mu_d-\mu_d\|_2$. The next result follows directly from \cite[Theorem 2.1]{al2025sharp} (with $H=\R^n$ and $\Sigma=I$).
\begin{proposition}[Concentration bound]\label{prop:mu4-ml}
Under Assumption~\ref{asmp:sg-4}, there is a constant $C>0$ (depending only on $d$ and the sub-Gaussian parameter $\sigma$) such that for all $\tau>0$,
\[
\Pr\!\left(
\|\widehat\mu_d-\mu_d\|_2
\ \le\
C\,\!\left(
\sqrt{\frac{n}{N}} \;+\; \frac{n^{d/2}}{N}+\sqrt{\frac{\tau}{N}}+\frac{\tau^{d/2}}{N}
\right)
\right)\ \ge\ 1-e^{-\tau}.
\]
\end{proposition}

The bounds on $\|\widehat\mu_d-\mu_d\|_2$ imply  bounds on the rotation recovery $\|\widehat Q-Q\|_{\rm F}$ in the objectives of our estimation problem in \eqref{eq:FU} and \eqref{eq:Ustar}. That is, since $\widehat \mu_d$ is close to $\mu_d$, the optima of $F_N$ should be close to those of $F$. Quantifying this relies on the Hessian of $F$. 

Extend $F$ to be a function on $\R^{n\times n}$. Its derivative $\D F(Q)$ at $Q\in \rO(n)$ is a linear functional on $\R^{n\times n}$. Similarly, the second derivative $\D^2 F(Q)$  is a linear map $\D^2 F(Q):\R^{n\times n}\to \R^{n\times n}$. Restricting $\D^2 F(Q)[\Delta]$ to $\Delta$ in the tangent space $U=T_Q\rO(n)$ and projecting the image to $U$, gives the Riemannian Hessian ${\rm Hess} F(Q):U\to U$. In other words,
\begin{equation*}
{\rm Hess} F(Q) [\Delta]\;:=\; \Pi_U (\D^2 F(Q)[\Delta])\qquad\mbox{for } \Delta\in U.    
\end{equation*}

\begin{proposition}[Parameter error in Frobenius norm]\label{prop:param-error}
Let $\widehat Q\in \mathrm O(n)$ be any empirical minimizer of $F_N$ in \eqref{eq:Ustar}. Assume there is a population minimizer $Q^\star$ of $F(Q)=\|g(Q)\|_{\rm F}^2$ such that the following curvature condition holds at every point of its orbit $Q^\star{\rm SP}(n)$:
there exists $\kappa>0$ with
\[
\langle \mathrm{Hess}F(Q)[\Delta],\ \Delta\rangle \ \ge\ \kappa\,\|\Delta\|_{\rm F}^2
\quad \text{for all }\Delta\in T_Q\mathrm O(n).
\]
Then, there exists $R\in{\rm SP}(n)$ such that
\begin{equation}\label{eq:mainFrobBound}
\|\widehat Q-Q^\star R\|_{\rm F}
\ \le\ \frac{C_d}{\kappa}\,\Big( d\,\|\mu_d\|_{\rm F} \;+\; d\,\| \widehat\mu_d-\mu_d\|_{\rm F} \Big)\ \| \widehat\mu_d-\mu_d\|_{\rm F},
\end{equation}
where $C_d>0$ depends only on $d$.
In particular, with probability at least $1-e^{-\tau}$,
\[
\|\widehat Q-Q^\star R\|_{\rm F}
\ \lesssim_d\ \frac{n^{\frac{d-1}{2}}}{\kappa}\,\|\mu_d\|_{\rm F}\,
\left(
\sqrt{\frac{n}{N}} \;+\; \frac{n^{d/2}}{N}\;+\;\sqrt{\frac{\tau}{N}}\;+\;\frac{\tau^{d/2}}{N}
\right),
\]
where the implicit constant depends only on $d$ and on the sub-Gaussian parameter of $\x$.
\end{proposition}

\section{Local versus global optima under pairwise mean independence} \label{sec:local_optima}

We now focus on a misspecified scenario: we assume $V=V_{\rm diag}^{d,n}$, but the true tensor lies in the larger space, $V_{\pmi}^{d,n}$. When $V=V_{\rm diag}$, 
minimizing $\Pi_{V^\perp}\!\bigl(Q^{\top}\bullet h_d(\s)\bigr)$ over $Q\in\rO(n)$
as in \eqref{eq:FU} is equivalent to 
\begin{equation}\label{eq:equivmax}
\mbox{maximize }\sum_{i=1}^n \bigl\langle h_d(\s),\ q_i^{\otimes d}\bigr\rangle^2,\qquad Q\in\rO(n),    
\end{equation}
where $q_1,\ldots,q_n$ are the columns of $Q$.
In the true ICA setting $h_d(\s)\in V_{\rm diag}$, this is maximized at matrices $Q$ such that $Q^\top\tilde A\in\rSP(n)$, by Lemma~\ref{lem:pop-ident}.
In other words, the procedure recovers the correct mixing matrix $\tilde A$ up to signed permutation.

To simplify notation we work at the population level and, without loss of generality, set $\tilde A=I$; i.e., $\x=\s$.  
With this convention, the identity $I$ (and any signed permutation $P\in\rSP(n)$) is a ground-truth rotation: sources are exactly recovered when the rows of $Q$ are the basis vectors $e_i$ up to sign.
We decompose
\[
h_d(\s)=\kappa =\kappa_{\diag}+\kappa_{\off},\qquad
\kappa_{\diag}=\sum_{a=1}^n \lambda_a e_a^{\otimes d},\quad \lambda_a=\kappa_d(s_a),
\]
where $\kappa_{\off}$ collects the PMI off-diagonal entries. That is, $\kappa_{\diag}$ = $\Pi_{V_{\diag}}(\kappa)$ and $\kappa_{\off}$ = $\Pi_{V_{\diag}^{\perp}}(\kappa)$
For $Q\in\rO(n)$ we write
\[
\Phi_{\diag}(Q):=\sum_{i=1}^n \bigl\langle \kappa_{\diag},\,q_i^{\otimes d}\bigr\rangle^2,
\]
so that the objective in \eqref{eq:equivmax} equals $\Phi_{\diag}(Q)$ when $\kappa_{\off}=0$.

We show that $I$ is a stationary point of \eqref{eq:equivmax}.
We write $Q=I+H$ with $H$ an infinitesimal perturbation in the tangent space to $\rO(n)$ at $I$ ($H+H^\top=0$, and hence $H$ has zeros on the diagonal). Then
\begin{align*}
((I+H)\bullet \kappa)_{i\cdots i}
&=\kappa_{i\cdots i}+\sum_{\ell\neq i}H_{i\ell}\kappa_{\ell i\cdots i}+\cdots+\sum_{\ell\neq i}H_{i\ell}\kappa_{i\cdots i\ell}+o(\|H\|_{\mathrm F}).
\end{align*}
If $\kappa\in V_{\pmi}$, then all terms with exactly one index $\ell\neq i$ vanish, so
\[
((I+H)\bullet \kappa)_{i\cdots i}^2-(I\bullet \kappa)_{i\cdots i}^2=o(\|H\|_{\mathrm F}).
\]
Hence $I$ (and every $P\in\rSP(n)$) is a stationary point of $F(Q)$ under PMI.  

Our simulations in Section~\ref{sec:numerical} show that this first order analysis can be misleading:  
as soon as $\|\kappa_{\off}\|_{\mathrm F}$ is moderately large, the global maximizer of $F$ may move away from $\rSP(n)$, and the landscape can develop spurious local maxima and saddle points.  

We now consider local versus global behavior for two random variables.
To make our first order analysis more explicit,
consider $n=2$ and $Q=R(\theta)$.  
Since $F$ is invariant under signed permutations, it suffices to study
\[
R(\theta)=\begin{pmatrix}\cos\theta&-\sin\theta\\ \sin\theta&\cos\theta\end{pmatrix}.
\]
Write $\kappa_{\diag}=\lambda_1 e_1^{\otimes d}+\lambda_2 e_2^{\otimes d}$ and denote
\[
\Phi_{\diag}(\theta)=\sum_{i=1}^2 \bigl\langle \kappa_{\diag}, q_i^{\otimes d}\bigr\rangle^2.
\]

When $K=K_{\diag}$ (no PMI terms), one computes
\begin{align*}
d=3:&\quad \Phi_{\diag}(\theta)=(\lambda_1^2+\lambda_2^2)\bigl(1-3\sin^2\theta\cos^2\theta\bigr),\\
d=4:&\quad \Phi_{\diag}(\theta)=(\lambda_1^2+\lambda_2^2)(1-4u+2u^2)+4t_1t_2 u^2,\qquad u=\sin^2\theta\cos^2\theta.
\end{align*}
In both cases $\Phi_{\diag}$ is globally maximized at $\theta=0,\pi/2$, i.e.\ at $\rSP(2)$.  
The drop away from $\rSP(2)$ is quadratic in $\theta$:
\begin{align}
 d=3:\quad & \Phi_{\diag}(0)-\Phi_{\diag}(\theta)=\tfrac{3}{4}(\lambda_1^2+\lambda_2^2)\sin^2(2\theta), \\
 d=4:\quad & \Phi_{\diag}(0)-\Phi_{\diag}(\theta) \ge \tfrac12\,\mathrm{gap}_4^{\,2}\sin^2(2\theta), \label{eq:gap}
\end{align}
where $\mathrm{gap}_4=|\lambda_1-\lambda_2|$.
For arbitrary $K_{\off}$, Cauchy–Schwarz gives
\begin{equation}\label{eq:kappa-off}
    F(R(\theta))\;\le\;\Phi_{\diag}(\theta)
+2\|\kappa_{\off}\|_{\mathrm F}\sum_{i=1}^2\bigl|\langle \kappa_{\diag},q_i^{\otimes d}\rangle\bigr|
+2\|\kappa_{\off}\|_{\mathrm F}^2.
\end{equation}
At $\theta=0$ (any $P\in\rSP(2)$), PMI zeros imply $F=\Phi_{\diag}(0)=\lambda_1^2+\lambda_2^2$.  
For $\theta\neq 0$, the diagonal drop is quadratic in $\theta$, while the perturbation from $\kappa_{\off}$ is linear.  
Thus for small $\theta$, $\rSP(2)$ are strict local maxima, but the global maximizer may drift away if $\|\kappa_{\off}\|_{\mathrm F}$ is not negligible.

This calculation illustrates the local versus global behavior under PMI.  
The ground-truth rotations $\rSP(n)$ are locally stable: infinitesimal deviations reduce the objective.  
But the global landscape can change once off-diagonal PMI terms are present, creating new optima far from $\rSP(n)$.  
This foreshadows the phenomena observed in our simulations.

\section{Numerical experiments} \label{sec:numerical}

In this section, we test the Riemannian gradient descent algorithm for PMICA introduced in \Cref{alg:cap} and compare to ICA approaches.
Consider the PMICA setup $\x = A 
\s$ with $A \in \rO(n)$, where we assume that the data comes prewhitened. We estimate $A$ by minimizing 
$F_N(Q) = \|\Pi_{V_{\pmi}^\perp}(Q^\top \bullet \widehat{\kappa}_4(\x))\|_{\rm F}^2$ over $\rO(n)$
using RGD. We call the approach \texttt{RGD-PMICA}.

We compare to ICA, where we minimize
$F_N^{\diag}(Q) = \|\Pi_{V_{\diag}^\perp}(Q^\top \bullet \widehat{\kappa}_4(\x))\|_{\rm F}^2$ over $\rO(n)$. We call this approach \texttt{RGD-ICA}
when we minimize using RGD. We compare these two RGD approaches with classical ICA baselines \texttt{FastICA} \cite{hyvarinen2000independent} and \texttt{JADE}  \cite{cardoso1993blind}. \texttt{FastICA} is a fixed-point algorithm that estimates independent components by maximizing non-Gaussianity, where orthogonality is maintained by whitening and a symmetric orthoognalization step at each iteration. \texttt{JADE} is a cumulant-based method that jointly diagonalizes fourth-order cumulant matrices via Jacobi (Givens) rotations, which enforce orthogonality by construction. As such, \texttt{JADE} is similar to \texttt{RGD-ICA} when $n=2$.

Having estimated a mixing matrix $\widehat{A} \in \rO(n)$, the ratio $\|\Pi_{V_{\diag}^\perp}({\widehat{A}}^\top \bullet \widehat{\kappa}_4(\x))\|_{\rm F}/\|\widehat{\kappa}_4(\x)\|_{\rm F}$ measures the goodness of fit of ICA, and $\|\Pi_{V_{\pmi}^\perp}({\widehat{A}}^\top \bullet \widehat{\kappa}_4(\x))\|_{\rm F}/\|\widehat{\kappa}_4(\x)\|_{\rm F}$ measures the goodness of fit of PMICA. We call these metrics ``Distance to Independent'' and ``Distance to PMI'', respectively. The denominator $\|\widehat{\kappa}_4(\x)\|_{\rm F}$ normalizes these metrics, ensuring that their values lie between zero and one.

\subsection{Synthetic data} 

Let $\z^{(0)}$, and $\z^{(1)}$ be two-dimensional random vectors independent from each other with density functions
\[
f_{\z^{(0)}}(z_1, z_2) = \frac{1}{2}|z_2| \mathbbm{1}_{[-1, 1]^2}(z_1, z_2), \quad \quad f_{\z^{(1)}}(z_1, z_2) = \frac{3}{2}|z_2| \mathbbm{1}_{B_1}(z_1, z_2),
\]
where $B_1 = \{(z_1, z_2) \in \R^2 \mid |z_1| + |z_2| \leq 1\}$ is the unit $\ell_1$ ball and $\mathbbm{1}_U$ is the indicator function for the set $U$. The vector $\z^{(0)}$ has independent entries and the entries of $\z^{(1)}$ are pairwise mean independent, which follows by direct calculations.
For each $\alpha \in [0,1]$, let $\z^{(\alpha)} = (1-\alpha)\z^{(0)} + \alpha\z^{(1)}$ and let $\s^{(\alpha)}$ be the random vector obtained by rescaling each coordinate of $\z^{(\alpha)}$ to have unit variance. See \Cref{fig:samples}.

\begin{figure}[ht]
    \centering
    \includegraphics[width=1\linewidth]{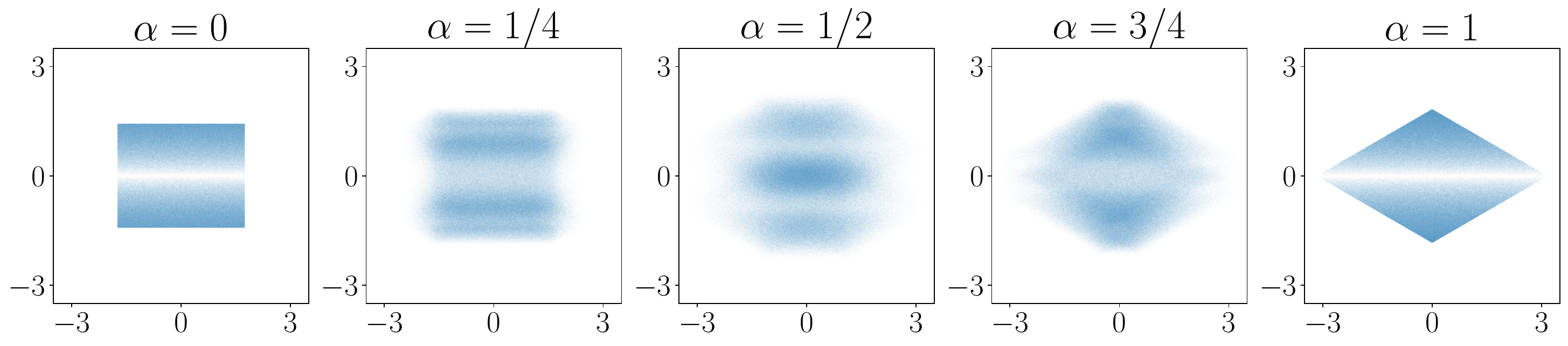}
    \caption{One million samples    from $\s^{(\alpha)}$ for different $\alpha$.  It is pairwise mean independent for all $\alpha$, because linear combinations of independent PMI vectors are PMI.  It is independent when $\alpha = 0$ but not otherwise.}
    \label{fig:samples}
\end{figure}

Given $A \in \rO(2)$ and $\alpha \in [0,1]$, consider the PMICA model $\x^{(\alpha)} = A \s^{(\alpha)}$.
We investigate the performance of the different algorithms in recovering $A$ given $N = 10^6$ samples from $\x^{(\alpha)}$,
see \Cref{fig:performance-ICA-synthetic}. In all plots in this section, each point is the median over 100 experiments. The fourth-order cumulant tensor $\kappa_4(\s^{(\alpha)}) \in S^4(\R^2)$ is generic in $V_{\pmi}$, in the sense of \Cref{thm:main-symmetric}, for all $\alpha \in [0,1]$. Hence, an algorithm recovers the true mixing matrix $A$ (up to signed permutation) if and only if the distance to PMI of $\widehat{A}^\top \bullet \widehat{\kappa}_4(\x)$ is zero.
\Cref{fig:performance-ICA-synthetic} shows that \texttt{RGD-PMICA} outperforms the ICA methods in recovering PMI distributions: ICA methods only recover the true sources when their distributions are close to independent (small $\alpha$).

\begin{figure}[ht]
    \centering
    \includegraphics[width=1\linewidth]{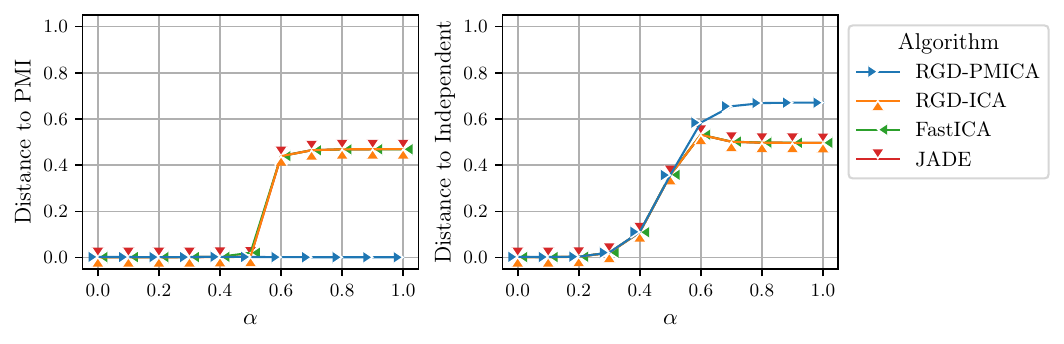}
    \caption{\texttt{RGD-PMICA} outperforms ICA algorithms in recovering $\s^{(\alpha)}$.    The ICA methods find the closest independent distribution, which is not the PMI one for $\alpha \geq 0.6$.
    }
    \label{fig:performance-ICA-synthetic}
\end{figure}

In \Cref{sec:local_optima}, we quantified what close to independent means for ICA to recover the true rotation: $\mathrm{gap}_4 (\widehat{\kappa}_4(\s)) = |\widehat{\kappa}_4(\s)_{1111} - \widehat{\kappa}_4(\s)_{2222}|$ has to be a significant amount larger than the off-diagonal Frobenius norm $\|\widehat{\kappa}_4(\s)_{\text{off}}\|_{\rm F} = \|\Pi_{V_{\diag}^\perp}(\widehat{\kappa}_4(\s))\|_{\rm F}$.  \Cref{fig:gap4} shows these quantities for different~$\s^{(\alpha)}$. When $\alpha = 0.6$, the ICA methods do not recover the true PMI source (as seen in \Cref{fig:performance-ICA-synthetic}), and the ratio between $\mathrm{gap}_4$ and $\|\widehat{\kappa}_4(\s)_{\off}\|$ is 1.07, c.f. \eqref{eq:gap} and~\eqref{eq:kappa-off}.

\begin{figure}[ht]
    \centering
    \includegraphics{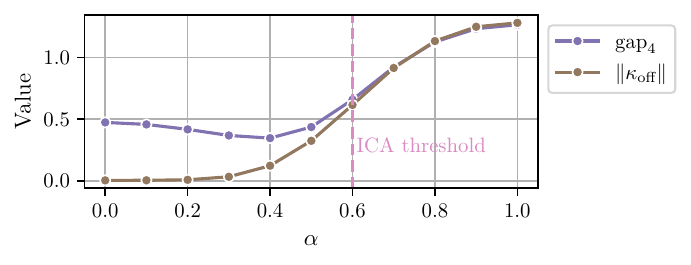}
    \caption{Gap ($|\widehat{\kappa}_{1111} - \widehat{\kappa}_{2222}|$) and off-diagonal Frobenius norm for $\widehat{\kappa}_4(\s^{(\alpha)})$. ICA models recover the PMI sources when $\alpha$ is smaller than the ICA threshold.}
    \label{fig:gap4}
\end{figure}

Now we study the performance of such algorithms in recovering PMI distributions for different dimensions.
Consider the distribution with density
\[
f(z_1, \dots, z_n; \alpha) = \frac{1}{C(\alpha)} \prod_{i=1}^n |z_i|^{\alpha_i - 1} \mathbbm{1}_{B_1^n} (z_1, \dots, z_n),
\]
where $C(\alpha)$ is a normalizing constant and $B_1^n = \{ y \in \R^n \mid \| y\|_1 \leq 1\}$.
This is closely related to the Dirichlet distribution. Let $\s$ be an $n$-dimensional random vector obtained from this distribution after rescaling each coordinate of $\z$ so that they have unit variance. Note that $\s$ is PMI for all $n$.
In the following, we use $\alpha_i = 2^{\frac{i-1}{n-1}}$. We choose these values of $\alpha_i$ to ensure that $\kappa_4(\s)$ is not close to $V_{\diag}$ and has sufficiently distinct diagonal entries. In particular, the fourth-order cumulant tensor $\kappa_4(\s)$ is generic in $V_{\pmi}$ for all $n$, in the sense of \Cref{thm:main-symmetric}. When $n=2$ this distribution coincides with the one defined above for $\s^{(1)}$. 

We evaluate the algorithms for different values of $n$ using a sample size of $10^6$. \Cref{fig:methods_vs_dimension} shows the best-case scenario for each method: the minimum after using each method with 25 random initializations. Again, \texttt{RGD-PMICA} outperforms the ICA algorithm in recovering PMI sources. In \Cref{fig:methods_vs_dimension}, the distance to independent is higher than the distance to PMI. This is expected since $V_{\diag} \subsetneq V_{\pmi}$.

\begin{figure}[h]
    \centering
    \includegraphics[width=\linewidth]{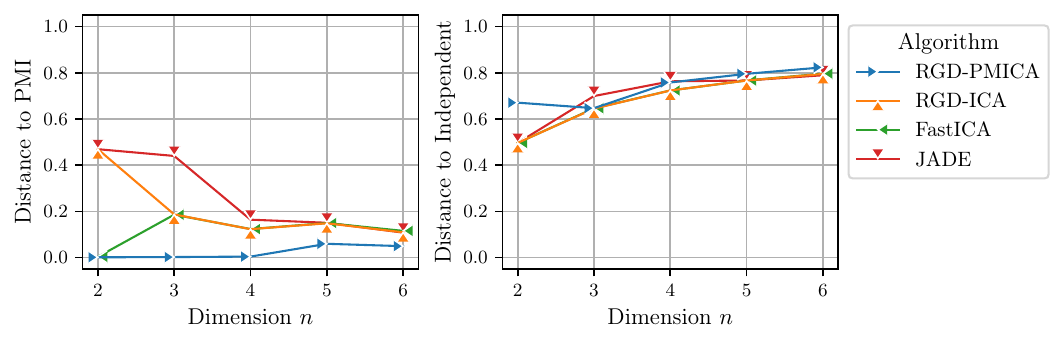}
    \caption{\texttt{RGD-PMICA} outperforms the ICA algorithms for all $n$. The decreasing trends of the ICA methods suggest that the closest independent distribution gets closer to the PMI one as $n$ increases.
    }
    \label{fig:methods_vs_dimension}
\end{figure}

In the remainder of this section, we focus on \texttt{RGD-PMICA}. \Cref{fig:methods_vs_dimension} shows that \texttt{RGD-PMICA} fails to recover the true mixing matrix when $n \geq 5$ with 25 different initializations. This is because the complexity of the optimization problem increases as $n$ grows. \Cref{fig:performance_trials} shows how the performance depends on the number of different random initializations. We observe that the number of initializations that \texttt{RGD-PMICA} requires to find the true PMI sources is of the order of $10^{n-2}$, where $n$ is the dimension.

\begin{figure}[ht]
    \centering
    \includegraphics{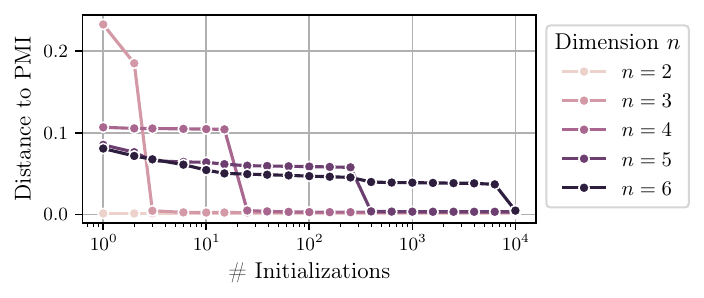}
    \caption{Convergence analysis of \texttt{RGD\_PMICA}. The complexity of the optimization problem grows exponentially with the dimension.}
    \label{fig:performance_trials}
\end{figure}

Finally, we analyze the sample complexity to recover PMI sources. \Cref{fig:sample_complexity} shows the standard deviation of $\widehat{\kappa}_4(\s)$ one million samples over 100 experiments, and the distance to PMI obtained with \texttt{RGD-PMICA} using $10^4$ initializations. We observe that the lines shared the same slope (provided the sample size is big enough). This slope is approximately $-1/2$ (in log-log scale), which agrees with the term $1/\sqrt{N}$ in \Cref{prop:mu4-ml}.

\begin{figure}[h]
    \centering
    \includegraphics[width=1\linewidth]{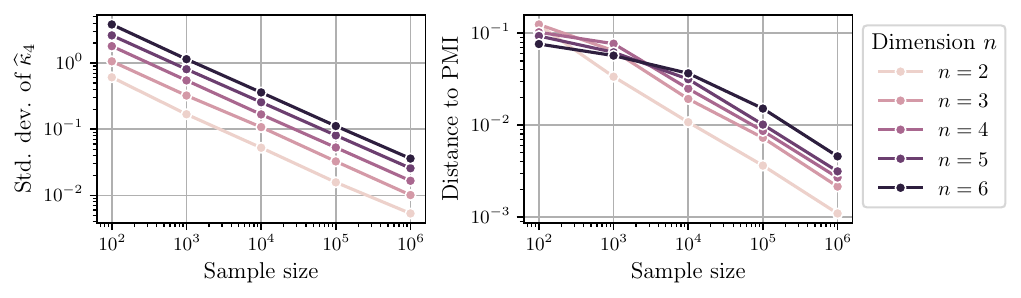}
    \caption{Sample complexity in different dimensions. Standard deviation of the estimator $\widehat{\kappa}_4(\mathbf{s})$ (left) and distance to PMI using \texttt{RGD-PMICA} (right). The lines corresponding to different dimensions share a common slope of $-1/2$.}
    \label{fig:sample_complexity}
\end{figure}

\subsection{Real data} 
In this section we test the performance of \texttt{RGD\_PMICA} on a real world application. The electroencephalogram (EEG) dataset \cite{chavarriaga2010learning} records the brain's electrical activity of six subjects over two sessions. We use the data indexed by S01-1 in the 22nd dataset available at \href{https://bnci-horizon-2020.eu/database/data-sets}{https://bnci-horizon-2020.eu/database/data-sets}, which consists of 91648 observations of 64 electrodes. First, we perform dimensionality reduction via PCA. We keep the top 5 principal components, which explain $93 \%$ variability of the data. This leads to a $91648 \times 5$ data matrix. After whitening the data, we test the performance of the different approaches mentioned above to recover independent/PMI sources. See \Cref{tab:accuracy_eeg}. \texttt{PCA} is the input data and \texttt{JADE} does not rely on any initialization. For each of the other optimization approaches and each metric, we report the best of 50 trials with different initial points in \Cref{tab:accuracy_eeg}.
\Cref{tab:accuracy_eeg} shows that \texttt{RGD-PMICA} outperforms the other methods in recovering PMI sources. It is interesting to observe that the distance to PMI achieved by \texttt{RGD-ICA} is approximately the one achieved by \texttt{RGD-PMICA}. This agrees with our discussion in \Cref{sec:local_optima}. However, the rotation found by \texttt{RGD-ICA} minimizing the distance to independent (0.55) leads to a distance to PMI of 0.18, i.e., it does not correspond to the one that minimizes the distance to PMI (0.10).

\begin{table}[h]
    \centering
    \begin{tabular}{l|c|c}
        Algorithm  & Distance to PMI & Distance to Independent\\
         \hline 
        \texttt{PCA} & 0.50 & 0.97 \\
        \texttt{RGD-PMICA} & $\mathbf{0.09}$ & $0.56$ \\
        \texttt{RGD-ICA} & $0.10$ & $0.55$ \\
        \texttt{FastICA} & $0.41$ & $0.69 $\\
        \texttt{JADE} & $0.34$ & $\mathbf{0.54}$\\
    \end{tabular}
    \caption{Goodness of fit of the PMICA and ICA models to EEG data} 
    \label{tab:accuracy_eeg}
\end{table}

A typical qualitative measurement to analyze the performance of ICA approaches in EEG data is the ability to retrieve a source corresponding to the eye-blink artifact. \Cref{fig:eeg_sources} shows that  \texttt{RGD-PMICA} succeeds in finding the eye-blink artifact.

\begin{figure}[h!]
    \centering
    \begin{subfigure}{1\linewidth}
        \includegraphics[width=1\linewidth]{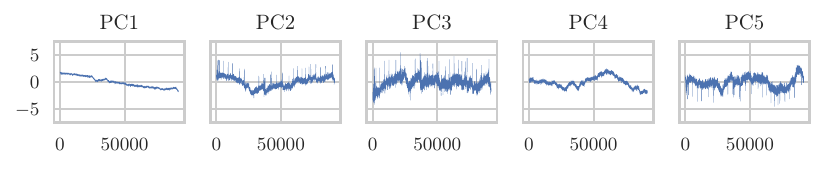}
        \caption{Projection onto principal components}
        \vspace*{2mm}
    \end{subfigure}
    \begin{subfigure}{1\linewidth}
        \includegraphics[width=1\linewidth]{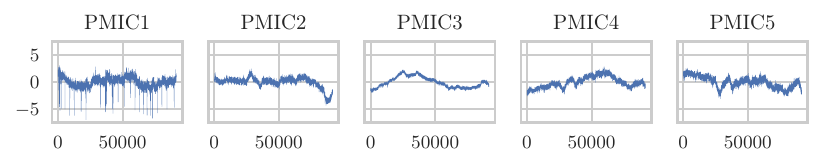}
        \caption{Projection onto pairwise mean independent components}
    \end{subfigure}
    \caption{
    \texttt{RGD-PMICA} finds the eye-blink artifact (PMIC1)     }
    \label{fig:eeg_sources}
\end{figure}

\section*{Discussion}

We established the identifiability of pairwise mean component analysis (PMICA), generalizing classical ICA by relaxing the independence assumption. It is more expressive than ICA, as it corresponds to a higher-dimensional family of cumulant tensors. Our results subsume previous examples of non-independent components analysis \cite{mesters2022non} and cannot be further relaxed: if we drop one mean independence assumption on the source variables, the model becomes unidentifiable.
We conclude by outlining directions for future work.

It is an open problem to determine the genericity conditions under which a symmetric tensor with an orthogonal basis of eigenvectors has a unique such basis. Our results in Section~\ref{sec:proof-uniqueness-symmetric} address symmetric tensors of order up to nine.

We focused on genericity conditions for PMICA, fixing the order of the cumulant tensor. It is an open question to characterize distributional assumptions for genericity when all cumulants are available---the analog of the `at most one Gaussian source' condition in ICA.
It is also an open problem to extend to the case where $\x$ and $\s$ have different dimensions, following the study of full independence in the overcomplete setting~\cite{eriksson2006complex,wang2024identifiability}.

Our concentration and parameter error bounds are universal in that they apply for any choice of $V$. We save a more refined analysis, particularly of improved moment and cumulant estimators for non-sub-Gaussian data, for future work.

\section*{Acknowledgments}
We thank Stanislav Volgushev and Zhekai Pang for helpful discussions. AR was supported by fellowships from ``la
Caixa'' Foundation (ID 100010434), with fellowship code LCF/BQ/EU23/12010097, and from RCCHU.
AS was supported by an Alfred P. Sloan research fellowship.
PZ is also affiliated with the Department of Statistical Sciences at the University of Toronto and he was
supported by NSERC grant RGPIN-2023-03481.  

\bibliographystyle{alpha}
\bibliography{biblio}

\newcommand{\etalchar}[1]{$^{#1}$}
\begin{thebibliography}{AGCSA25}

\bibitem[AGCSA25]{al2025sharp}
Omar Al-Ghattas, Jiaheng Chen, and Daniel Sanz-Alonso.
\newblock Sharp concentration of simple random tensors.
\newblock {\em arXiv preprint arXiv:2502.16916}, 2025.

\bibitem[AGH{\etalchar{+}}14]{anandkumar2014tensor}
Animashree Anandkumar, Rong Ge, Daniel~J Hsu, Sham~M Kakade, and Matus Telgarsky.
\newblock Tensor decompositions for learning latent variable models.
\newblock {\em J. Mach. Learn. Res.}, 15(1):2773--2832, 2014.

\bibitem[Akh20]{akhiezer2020classical}
Naum~I. Akhiezer.
\newblock {\em The Classical Moment Problem and Some Related Questions in Analysis}, volume~5 of {\em Classics in Applied Mathematics}.
\newblock SIAM, Philadelphia, PA, 2020.
\newblock Reprint of the 1965 English edition, originally published by Oliver and Boyd.

\bibitem[AMS08]{Absil2008}
P.-A. Absil, R.~Mahony, and R.~Sepulchre.
\newblock {\em Optimization Algorithms on Matrix Manifolds}.
\newblock Princeton University Press, Princeton, NJ, 2008.

\bibitem[Bil95]{billingsley1995probability}
Patrick Billingsley.
\newblock {\em Probability and Measure}.
\newblock Wiley Series in Probability and Mathematical Statistics. John Wiley \& Sons, New York, 3rd edition, 1995.

\bibitem[BW97]{back1997first}
Andrew~D Back and Andreas~S Weigend.
\newblock A first application of independent component analysis to extracting structure from stock returns.
\newblock {\em International journal of neural systems}, 8(04):473--484, 1997.

\bibitem[Car98]{cardoso1998blind}
J-F Cardoso.
\newblock Blind signal separation: statistical principles.
\newblock {\em Proceedings of the IEEE}, 86(10):2009--2025, 1998.

\bibitem[CGLM08]{comon2008symmetric}
Pierre Comon, Gene Golub, Lek-Heng Lim, and Bernard Mourrain.
\newblock Symmetric tensors and symmetric tensor rank.
\newblock {\em SIAM Journal on Matrix Analysis and Applications}, 30(3):1254--1279, 2008.

\bibitem[CM10]{chavarriaga2010learning}
Ricardo Chavarriaga and Jos{\'e} del~R Mill{\'a}n.
\newblock Learning from {EEG} error-related potentials in noninvasive brain-computer interfaces.
\newblock {\em IEEE transactions on neural systems and rehabilitation engineering}, 18(4):381--388, 2010.

\bibitem[Com94]{comon1994independent}
Pierre Comon.
\newblock Independent component analysis, a new concept?
\newblock {\em Signal processing}, 36(3):287--314, 1994.

\bibitem[CS93]{cardoso1993blind}
Jean-Fran{\c{c}}ois Cardoso and Antoine Souloumiac.
\newblock Blind beamforming for non-{G}aussian signals.
\newblock In {\em IEE proceedings F (radar and signal processing)}, volume 140, pages 362--370. IET, 1993.

\bibitem[CS13]{cartwright2013number}
Dustin Cartwright and Bernd Sturmfels.
\newblock The number of eigenvalues of a tensor.
\newblock {\em Linear algebra and its applications}, 438(2):942--952, 2013.

\bibitem[EK06]{eriksson2006complex}
Jan Eriksson and Visa Koivunen.
\newblock Complex random vectors and {ICA} models: Identifiability, uniqueness, and separability.
\newblock {\em IEEE Transactions on Information theory}, 52(3):1017--1029, 2006.

\bibitem[Fam70]{fama1970efficient}
Eugene~F. Fama.
\newblock Efficient capital markets: A review of theory and empirical work.
\newblock {\em Journal of Finance}, 25(2):383--417, May 1970.

\bibitem[Fri13]{friedland2013best}
Shmuel Friedland.
\newblock Best rank one approximation of real symmetric tensors can be chosen symmetric.
\newblock {\em Frontiers of Mathematics in China}, 8(1):19--40, 2013.

\bibitem[GLS24]{garrote2024cumulant}
Marina Garrote-L{\'o}pez and Monroe Stephenson.
\newblock Cumulant tensors in partitioned independent component analysis.
\newblock {\em arXiv preprint arXiv:2402.10089}, 2024.

\bibitem[GS]{M2}
Daniel~R. Grayson and Michael~E. Stillman.
\newblock Macaulay2, a software system for research in algebraic geometry.

\bibitem[Han82]{hansen1982gmm}
Lars~Peter Hansen.
\newblock Large sample properties of generalized method of moments estimators.
\newblock {\em Econometrica}, 50(4):1029--1054, 1982.

\bibitem[HHI01]{hyvarinen2001topographic}
Aapo Hyv{\"a}rinen, Patrik~O Hoyer, and Mika Inki.
\newblock Topographic independent component analysis.
\newblock {\em Neural computation}, 13(7):1527--1558, 2001.

\bibitem[HO00]{hyvarinen2000independent}
Aapo Hyv{\"a}rinen and Erkki Oja.
\newblock Independent component analysis: algorithms and applications.
\newblock {\em Neural networks}, 13(4-5):411--430, 2000.

\bibitem[IR15]{imbens2015causal}
Guido~W. Imbens and Donald~B. Rubin.
\newblock {\em Causal Inference for Statistics, Social, and Biomedical Sciences: An Introduction}.
\newblock Cambridge University Press, Cambridge, 2015.

\bibitem[Jia25]{jiang2025identification}
Ziyu Jiang.
\newblock Identification and estimation of simultaneous equation models using higher-order cumulant restrictions.
\newblock {\em arXiv:2501.06777}, 2025.

\bibitem[Kel70]{kelker1970distribution}
Douglas Kelker.
\newblock Distribution theory of spherical distributions and a location-scale parameter generalization.
\newblock {\em Sankhy{\=a}: The Indian Journal of Statistics, Series A}, pages 419--430, 1970.

\bibitem[LCMS25]{carreno2024linear}
Paula Leyes~Carreno, Chiara Meroni, and Anna Seigal.
\newblock Linear causal disentanglement via higher-order cumulants.
\newblock {\em La Matematica}, pages 1--40, 2025.

\bibitem[Lim05]{lim2005singular}
Lek-Heng Lim.
\newblock Singular values and eigenvalues of tensors: a variational approach.
\newblock In {\em 1st IEEE International Workshop on Computational Advances in Multi-Sensor Adaptive Processing, 2005.}, pages 129--132. IEEE, 2005.

\bibitem[LZS20]{lee2020testing}
CE~Lee, X~Zhang, and Xiaofeng Shao.
\newblock Testing conditional mean independence for functional data.
\newblock {\em Biometrika}, 107(2):331--346, 2020.

\bibitem[Mar39]{marcinkiewicz1939propriete}
J{\'o}zef Marcinkiewicz.
\newblock Sur une propri{\'e}t{\'e} de la loi de {G}auss.
\newblock {\em Mathematische Zeitschrift}, 44(1):612--618, 1939.

\bibitem[MBJS95]{makeig1995independent}
Scott Makeig, Anthony Bell, Tzyy-Ping Jung, and Terrence~J Sejnowski.
\newblock Independent component analysis of electroencephalographic data.
\newblock {\em Advances in neural information processing systems}, 8, 1995.

\bibitem[McC18]{mccullagh2018tensor}
Peter McCullagh.
\newblock {\em Tensor methods in statistics: Monographs on statistics and applied probability}.
\newblock Chapman and Hall/CRC, 2018.

\bibitem[MZ24]{mesters2022non}
Geert Mesters and Piotr Zwiernik.
\newblock Non-independent components analysis.
\newblock {\em The Annals of Statistics}, 52(6):2506--2528, 2024.

\bibitem[Qi11]{qi2011best}
Liqun Qi.
\newblock The best rank-one approximation ratio of a tensor space.
\newblock {\em SIAM Journal on matrix analysis and applications}, 32(2):430--442, 2011.

\bibitem[Rob16]{robeva2016orthogonal}
Elina Robeva.
\newblock Orthogonal decomposition of symmetric tensors.
\newblock {\em SIAM Journal on Matrix Analysis and Applications}, 37(1):86--102, 2016.

\bibitem[Rom05]{roman2005advanced}
Steven Roman.
\newblock {\em Advanced linear algebra}, volume~3.
\newblock Springer, 2005.

\bibitem[RSZ25]{ribot2025orthogonal}
Alvaro Ribot, Anna Seigal, and Piotr Zwiernik.
\newblock Orthogonal eigenvectors and singular vectors of tensors.
\newblock {\em arXiv preprint arXiv:2506.19009}, 2025.

\bibitem[RZ21]{rossell2021dependence}
David Rossell and Piotr Zwiernik.
\newblock Dependence in elliptical partial correlation graphs.
\newblock {\em Electronic Journal of Statistics}, 15(2):4236--4263, 2021.

\bibitem[Woo95]{wooldridge1995selection}
Jeffrey~M Wooldridge.
\newblock Selection corrections for panel data models under conditional mean independence assumptions.
\newblock {\em Journal of econometrics}, 68(1):115--132, 1995.

\bibitem[Woo10]{wooldridge2010econometric}
Jeffrey~M Wooldridge.
\newblock {\em Econometric analysis of cross section and panel data}.
\newblock MIT press, 2010.

\bibitem[WS24]{wang2024identifiability}
Kexin Wang and Anna Seigal.
\newblock Identifiability of overcomplete independent component analysis.
\newblock {\em arXiv:2401.14709}, 2024.

\bibitem[Zwi15]{zwiernik2015semialgebraic}
Piotr Zwiernik.
\newblock {\em Semialgebraic statistics and latent tree models}.
\newblock CRC Press, 2015.

\end{thebibliography}

\markboth{APPENDIX}{APPENDIX}
\appendix

\section{Technical material from  Section~\ref{sec:finitebounds}}\label{sec:appendix-finite-bounds}

\subsection{The Gradient and Hessian of $F, F_N$}

The functions $F$ and $F_N$ are defined in \eqref{eq:FU} and \eqref{eq:Ustar}. The gradient $\D F(Q)$ of $F$ at $Q$ is the  linear functional $\nabla F(Q)$: $\D F(Q)[\Delta]=\<\nabla F(Q),\Delta\>$ on $\R^{n\times n}$. Let ${\rm grad} F$ denote its projection to  the tangent space to $\rO(n)$ at $Q$
\[
T_Q\mathrm O(n)\;=\;\{\Delta\in\R^{n\times n}: Q^\top\Delta+\Delta^\top Q=0\}
=\{QA:\ A^\top=-A\}.
\]
Fixing $Q$, we denote $U:=T_Q\rO(n)$. The Riemannian gradient of $F$ at $Q$ is
$$
{\rm grad} F(Q)\;:=\;\Pi_U(\nabla F(Q)).
$$
Since, for any $\Delta\in \R^{n\times n}$
$$
\<{\rm grad}F(Q),\Delta\>\;=\;\<{\rm grad}F(Q),\Pi_U(\Delta)\>\;=\;\<\nabla F(Q),\Pi_U(\Delta)\>
$$
this restricts $\D F(Q)$ to $\Delta\in U$; see \cite[Chapter~3]{Absil2008}.
We express the gradients of $F,F_N$ in terms of $g,g_N$.
\begin{lemma}\label{lem:A1}
    Let $Q\in \rO(n)$ and let $U=T_Q\rO(n)$ then $${\rm grad} F(Q)=2\Pi_U\Big((\D g(Q))^\ast[g(Q)]\Big)\quad\mbox{and}\quad{\rm grad} F_N(Q)=2\Pi_U\Big((\D g_N(Q))^\ast[g_N(Q)]\Big), $$
    where for a linear mapping $G$ on $\R^{n\times n}$, $G^*$ denotes its conjugate mapping.
\end{lemma}
\begin{proof}
From the definition of the derivative $\D g(Q)$ of $g$ at $Q$, we get
$$g(Q+\Delta)-g(Q)\;=\;\D g(Q)[\Delta]+o(\|\Delta\|).$$ 
It follows that 
$$
F(Q+\Delta)-F(Q)\;=\;2\<g(Q),\D g(Q)[\Delta]\>+o(\|\Delta\|)\;=\;2\<(\D g(Q))^\ast[g(Q)],\Delta\>+o(\|\Delta\|).
$$
 Hence the derivative $\D F(Q)$ is represented by the gradient $\nabla F(Q)$ which satisfies
\(
\nabla F(Q)\ =\ 2\, (\D g(Q))^\ast[g(Q)]\;\in \;\R^{n\times n}\).
The Riemannian gradient is its orthogonal projection to $U$. The argument for $F_N$ is analogous.
\end{proof}

\begin{lemma}\label{lem:A3}
For any $Q\in\rO(n)$,
\[
\|\D g_N(Q)-\D g(Q)\| \ \le\ d\,\|\widehat\mu_d-\mu_d\|_{\rm F},
\qquad
\|\D g(Q)\|\ \le\ d\,\|\mu_d\|_{\rm F},
\]
where $\|\cdot\|$ is the operator norm induced by matrix $\|\cdot\|_{\rm F}$ and tensor $\|\cdot\|_{\rm F}$.
\end{lemma}
\begin{proof}
Let $\T:=\widehat\mu_d-\mu_d$ and $H(Q):=Q^\top\bullet \T$ so that
$\D g_N(Q)-\D g(Q)=\Pi_{V^\perp}(\D H(Q))$.
For any $\Delta\in\R^{n\times n}$,
\[
\D H(Q)[\Delta]
=\sum_{j=1}^d
\bigl(Q^\top\bigr)^{\otimes(j-1)}\otimes \Delta^\top \otimes \bigl(Q^\top\bigr)^{\otimes(d-j)}\ \bullet\ \T.
\]
By multilinearity and $\|Q\|=1$,
\[
\|\D H(Q)[\Delta]\|_{\rm F}
\ \le\ \sum_{j=1}^d \|\Delta\|\,\|\T\|_{\rm F}
\ \le\ d\,\|\Delta\|_{\rm F}\,\|\T\|_{\rm F}.
\]
Thus $\|\D H(Q)\|\le d\,\|\T\|_{\rm F}$. Since $\Pi_{V^\perp}$ is an orthogonal projection,
$\|\Pi_{V^\perp}\|=1$, hence
$\|\D g_N(Q)-\D g(Q)\|\le d\,\|\T\|_{\rm F}$.
The bound for $\D g(Q)$ follows by the same argument with $\T$ replaced by $\mu_d$.
\end{proof}

\subsection{The spectral norm of a tensor}

Let $\T\in S^d(\R^n)$. We collect useful results on the
spectral norm
\[
\|\T\|_2=\max_{\|u\|=1} \<\T,u^{\otimes d}\>,
\]
defined in \eqref{eq:spnorm}.  One can view $T$ as a multilinear form and define 
\[
\|\T\|_{2, \dots, 2} \;:=\; \max_{\| u^{(1)}\|=\cdots =\|u^{(d)}\|=1}
\<\T, u^{(1)}\otimes\cdots\otimes u^{(d)}\>.
\]
The norm
$\|\T\|_{2,\dots,2}$ (resp. $\|\T\|_2$) characterizes the best
rank-one approximation (resp. symmetric approximation) of $\T$.
For general tensors the two norms differ, but for symmetric tensors the best rank-one approximation may be
chosen symmetric~\cite[Theorem~1]{friedland2013best}. Hence  
\begin{equation}\label{eq:tensbound}
\|\T\|_{2,\ldots,2}\;=\;\|\T\|_2
\end{equation}
for all $\T\in S^d(\R^n)$.
The next lemma compares the Frobenius norm
and the spectral norm.

\begin{lemma}\label{lem:frobinml}
   For any $\T\in S^d(\R^n)$, we have 
   \(
   \|\T\|_{\rm F}\;\leq\;n^{(d-1)/2}\,\|\T\|_2\).
\end{lemma}

\begin{proof}
For $d=2,3$ this appears in
\cite[Theorem~4.3]{qi2011best}. The author conjectured that the same bound
holds for all $d\geq 4$ \cite[Conjecture~2]{qi2011best}. As explained in
the discussion following that conjecture, his Conjecture~1 together with
\cite[Theorem~3.1]{qi2011best} imply the result. Conjecture~1 is \eqref{eq:tensbound}, which was established in
\cite[Theorem~1]{friedland2013best}.
\end{proof}

\subsection{Proof of Proposition~\ref{prop:param-error}}

By first-order optimality,
\[
\mathrm{grad}\,F_N(\widehat Q)=0,
\qquad
\mathrm{grad}\,F(Q)=0 \quad \text{for every } Q\in Q^\star{\rm SP}(n).
\]
Let $Q_0\in Q^\star{\rm SP}(n)$ be the orbit point minimizing
$\|\widehat Q-Q_0\|_{\rm F}$, and let
$\gamma:[0,1]\to\mathrm O(n)$ be the geodesic from $Q_0$ to $\widehat Q$
with tangent $\Delta=\dot\gamma(0)\in T_{Q_0}\mathrm O(n)$.
A Taylor expansion of the empirical gradient along $\gamma$
 gives
\[
0=\mathrm{grad}\,F_N(\widehat Q)
=\mathrm{grad}\,F_N(Q_0)+\mathrm{Hess}F_N(Q_0)[\Delta]+o(\|\Delta\|_{\rm F}).
\]
Since $\mathrm{grad}\,F(Q_0)=0$, we obtain
\[
\underbrace{\bigl(\mathrm{grad}\,F_N(Q_0) \! - \! \mathrm{grad}\,F(Q_0)\bigr)}_{\text{score fluctuation}}
\!+\!\underbrace{\mathrm{Hess}F(Q_0)[\Delta]}_{\text{population curvature}}
\!+\!\underbrace{\bigl(\mathrm{Hess}F_N(Q_0) \! - \! \mathrm{Hess}F(Q_0)\bigr)[\Delta]}_{\text{Hessian fluctuation}}
\!+ o(\|\Delta\|_{\rm F})=\!0.
\]

Both fluctuation terms are linear in 
$\T=\widehat\mu_d-\mu_d$. Indeed, at $Q_0$ we have $g(Q_0)=0$, so, using Lemma~\ref{lem:A1},
\begin{align*}
\mathrm{grad}\,F_N(Q_0)-\mathrm{grad}\,F(Q_0)
& =2\Pi_U\Big(\bigl(Dg_N(Q_0)^\ast [g_N(Q_0)]-Dg(Q_0)^\ast [g(Q_0)]\bigr)\Big)\\ 
& =2\,\Pi_U\Big(Dg_N(Q_0)^\ast\![g_N(Q_0)-g(Q_0)]\Big).
\end{align*} 
As before,
\[
\|g_N(Q_0)-g(Q_0)\|_{\rm F}
=\|\Pi_{V^\perp}(Q_0^{\top}\!\bullet(\widehat\mu_d-\mu_d))\|_{\rm F}
\le \|\widehat\mu_d-\mu_d\|_{\rm F}.
\]
Insert and subtract $Dg(Q_0)$:
\[
Dg_N(Q_0)^\ast\,[g_N-g]
= Dg(Q_0)^\ast\,[g_N-g]
+\big(Dg_N(Q_0)-Dg(Q_0)\big)^\ast\,[g_N-g].
\]
Taking Frobenius norms and using Lemma~\ref{lem:A3} together with
$\|g_N(Q_0)-g(Q_0)\|_{\rm F}\le \|\widehat\mu_d-\mu_d\|_{\rm F}$, we obtain
\begin{align*}
\|\mathrm{grad}\,F_N(Q_0)-\mathrm{grad}\,F(Q_0)\|_{\rm F}
&\le 2\,\|Dg(Q_0)\|\,\|g_N(Q_0)-g(Q_0)\|_{\rm F}\\
&\quad + 2\,\|Dg_N(Q_0)-Dg(Q_0)\|\,\|g_N(Q_0)-g(Q_0)\|_{\rm F}\\
&\le 2d\,\|\mu_d\|_{\rm F}\,\|\widehat\mu_d-\mu_d\|_{\rm F}
\;+\; 2d\,\|\widehat\mu_d-\mu_d\|_{\rm F}^{\,2}.
\end{align*}
Similarly, differentiating $\mathrm{grad}\,F_N(Q)=2\,Dg_N(Q)^\ast g_N(Q)$ and
using $g(Q_0)=0$ gives
\[
\mathrm{Hess}F_N(Q_0)[\Delta]
=2\,Dg_N(Q_0)^\ast Dg_N(Q_0)[\Delta],
\qquad
\mathrm{Hess}F(Q_0)[\Delta]
=2\,Dg(Q_0)^\ast Dg(Q_0)[\Delta].
\]
We write
\[
\mathrm{Hess}F_N(Q_0)-\mathrm{Hess}F(Q_0)
=2\Big( (Dg_N-Dg)^\ast Dg + Dg^\ast (Dg_N-Dg) + (Dg_N-Dg)^\ast (Dg_N-Dg)\Big),
\]
all evaluated at $Q_0$.
Hence, by Lemma~\ref{lem:A3},
\begin{align*}
\|\mathrm{Hess}F_N(Q_0)-\mathrm{Hess}F(Q_0)\|
&\le 4\|Dg(Q_0)\|\|Dg_N(Q_0)-Dg(Q_0)\|
+ 2\|Dg_N(Q_0)-Dg(Q_0)\|^2\\
&\le 4d^2\,\|\mu_d\|_{\rm F}\,\|\widehat\mu_d-\mu_d\|_{\rm F}
+ 2d^2\,\|\widehat\mu_d-\mu_d\|_{\rm F}^{\,2}.
\end{align*}
Thus, up to a quadratic remainder, both fluctuations are linear in $\|\widehat\mu_d-\mu_d\|_{\rm F}$ with constants depending only on $d$ and $\|\mu_d\|_{\rm F}$.
Invoking the curvature condition on $T_{Q_0}\mathrm O(n)$,
\[
\langle \mathrm{Hess}F(Q_0)[\Delta],\Delta\rangle \ \ge\ \kappa\,\|\Delta\|_{\rm F}^2,
\]
and using the Taylor expansion above, we obtain (for $\|\widehat\mu_d-\mu_d\|_{\rm F}$ small enough),
\[
\kappa\,\|\Delta\|_{\rm F}
\ \lesssim\
\|\mathrm{grad}\,F_N(Q_0)-\mathrm{grad}\,F(Q_0)\|_{\rm F}
\;+\;\|\mathrm{Hess}F_N(Q_0)-\mathrm{Hess}F(Q_0)\|\,\|\Delta\|_{\rm F}.
\]
Absorbing the last term into the left-hand side and using the bounds above, we get
\[
\|\Delta\|_{\rm F}
\ \lesssim\ \frac{d\,\|\mu_d\|_{\rm F}}{\kappa}\,\|\widehat\mu_d-\mu_d\|_{\rm F}
\;+\; \frac{d}{\kappa}\,\|\widehat\mu_d-\mu_d\|_{\rm F}^{\,2}.
\]
The tangent space at $Q_0\in \mathrm O(n)$ is
$T_{Q_0}\mathrm O(n)=\{Q_0\Omega:\;\Omega^\top=-\Omega\}$.
For any $\widehat Q$ sufficiently close to $Q_0$, there exists a unique skew-symmetric
$\Omega$ such that
\[
\widehat Q = Q_0 \exp(\Omega).
\]
The geodesic $\gamma(t)=Q_0\exp(t\Omega)$ then connects $Q_0$ to $\widehat Q$,
with initial tangent $\Delta=\dot\gamma(0)=Q_0\Omega\in T_{Q_0}\mathrm O(n)$.
Since multiplication by $Q_0$ is an isometry,
\[
\|\Delta\|_{\rm F} = \|\Omega\|_{\rm F}.
\]
Moreover,
\[
\widehat Q - Q_0 = Q_0\big(\exp(\Omega)-I\big)
= Q_0\Big(\Omega+\tfrac{1}{2}\Omega^2+\cdots\Big),
\]
so
\[
\|\widehat Q-Q_0\|_{\rm F} \;\le\; \|\Omega\|_{\rm F} + O(\|\Omega\|_{\rm F}^2)
= \|\Delta\|_{\rm F} + O(\|\Delta\|_{\rm F}^2).
\]
Thus, for sufficiently small $\|\Delta\|_{\rm F}$,
\[
\|\widehat Q-Q_0\|_{\rm F} \;\lesssim\; \|\Delta\|_{\rm F}.
\]
Combining with the bound on $\|\Delta\|_{\rm F}$ obtained above yields
\[
\|\widehat Q-Q_0\|_{\rm F}
\ \lesssim\ \frac{d\,\|\mu_d\|_{\rm F}}{\kappa}\,\|\widehat\mu_d-\mu_d\|_{\rm F}
\;+\; \frac{d}{\kappa}\,\|\widehat\mu_d-\mu_d\|_{\rm F}^{\,2}.
\]
Since $Q_0\in Q^\star{\rm SP}(n)$ was chosen to minimize $\|\widehat Q-Q\|_{\rm F}$,
the result follows.

The above argument shows that inequality
\eqref{eq:mainFrobBound} holds deterministically, given a realization of
$\widehat\mu_d$. By Proposition~\ref{prop:mu4-ml}, the deviation
$\|\widehat\mu_d-\mu_d\|_2$ admits a sharp concentration bound with
probability at least $1-e^{-\tau}$. Using Lemma~\ref{lem:frobinml} to
relate Frobenius and spectral norms then yields the high-probability
bound stated in Proposition~\ref{prop:param-error}.

\end{document}